\documentclass[12pt]{amsart}
\usepackage{}
\usepackage{amsmath}
\usepackage{amsfonts}
\usepackage{amssymb}
\usepackage[all]{xy}           

\usepackage{bbding}
\usepackage{txfonts}
\usepackage[shortlabels]{enumitem}
\usepackage{ifpdf}
\ifpdf
  \usepackage[colorlinks,final,backref=page,hyperindex]{hyperref}
\else
  \usepackage[colorlinks,final,backref=page,hyperindex,hypertex]{hyperref}
\fi
\usepackage{tikz}
\usepackage[active]{srcltx}

\makeatletter

\newtheorem{theorem}{Theorem}[section]
\newtheorem{prop}[theorem]{Proposition}
\newtheorem{lemma}[theorem]{Lemma}
\newtheorem{coro}[theorem]{Corollary}
\newtheorem{prop-def}{Proposition-Definition}[section]

\newtheorem{conjecture}[theorem]{Conjecture}

\newtheorem{problem}{Problem}

\theoremstyle{definition} 
\newtheorem{defn}[theorem]{Definition}

\newtheorem{exam}[theorem]{Example}

\newcommand{\nc}{\newcommand}

\nc{\delete}[1]{{}}
\nc{\mmargin}[1]{}

\nc{\mlabel}[1]{\label{#1}}  
\nc{\mcite}[1]{\cite{#1}}  
\nc{\mref}[1]{\ref{#1}}  
\nc{\mbibitem}[1]{\bibitem{#1}} 

\delete{
\nc{\mlabel}[1]{\label{#1}  
{\hfill \hspace{1cm}{\bf{{\ }\hfill(#1)}}}}
\nc{\mcite}[1]{\cite{#1}{{\bf{{\ }(#1)}}}}  
\nc{\mref}[1]{\ref{#1}{{\bf{{\ }(#1)}}}}  
\nc{\mbibitem}[1]{\bibitem[\bf #1]{#1}} 
}

\nc{\difforg}{U_D}
\nc{\rbforg}{U_{RB}}
\nc{\multforg}{U_{M}}
\nc{\diffree}{{}_DF}
\nc{\difcof}{F_D}
\nc{\rbfree}{{}_{RB}F}
\nc{\rbcof}{F_{RB}}
\nc{\multfree}{{}_{M}F}
\nc{\multcof}{F_M}

\nc{\vep}{\varepsilon}
\nc{\bin}[2]{ (_{\stackrel{\scs{#1}}{\scs{#2}}})}  
\nc{\binc}[2]{(\!\! \begin{array}{c} \scs{#1}\\
    \scs{#2} \end{array}\!\!)}  
\nc{\bincc}[2]{  ( {\scs{#1} \atop
    \vspace{-1cm}\scs{#2}} )}  
\nc{\bs}{\bar{S}}
\nc{\la}{\longrightarrow}
\nc{\ot}{\otimes}
\nc{\rar}{\rightarrow}
\nc{\dar}{\downarrow}
\nc{\dap}[1]{\downarrow \rlap{$\scriptstyle{#1}$}}
\nc{\defeq}{\stackrel{\rm def}{=}}
\nc{\dis}[1]{\displaystyle{#1}}
\nc{\dotcup}{\ \displaystyle{\bigcup^\bullet}\ }
\nc{\hcm}{\ \hat{,}\ }
\nc{\hts}{\hat{\otimes}}
\nc{\hcirc}{\hat{\circ}}
\nc{\lleft}{[}
\nc{\lright}{]}
\nc{\curlyl}{\left \{ \begin{array}{c} {} \\ {} \end{array}
    \right .  \!\!\!\!\!\!\!}
\nc{\curlyr}{ \!\!\!\!\!\!\!
    \left . \begin{array}{c} {} \\ {} \end{array}
    \right \} }
\nc{\longmid}{\left | \begin{array}{c} {} \\ {} \end{array}
    \right . \!\!\!\!\!\!\!}
\nc{\ora}[1]{\stackrel{#1}{\rar}}
\nc{\ola}[1]{\stackrel{#1}{\la}}
\nc{\scs}[1]{\scriptstyle{#1}} \nc{\mrm}[1]{{\rm #1}}
\nc{\dirlim}{\displaystyle{\lim_{\longrightarrow}}\,}
\nc{\invlim}{\displaystyle{\lim_{\longleftarrow}}\,}
\nc{\dislim}[1]{\displaystyle{\lim_{#1}}} \nc{\colim}{\mrm{colim}}
\nc{\mvp}{\vspace{0.3cm}} \nc{\tk}{^{(k)}} \nc{\tp}{^\prime}
\nc{\ttp}{^{\prime\prime}} \nc{\svp}{\vspace{2cm}}
\nc{\vp}{\vspace{8cm}}
\nc{\modg}[1]{\!<\!\!{#1}\!\!>}
\nc{\intg}[1]{F_C(#1)}
\nc{\lmodg}{\!<\!\!}
\nc{\rmodg}{\!\!>\!}
\nc{\cpi}{\widehat{\Pi}}
\nc{\sha}{{\mbox{\cyr X}}}  
\nc{\ssha}{{\mbox{\cyrs X}}} 
\nc{\tsha}{{\mbox{\cyrt X}}}
\nc{\shpr}{\diamond}    
\nc{\labs}{\mid\!}
\nc{\rabs}{\!\mid}

\font\cyr=wncyr10
\font\cyrs=wncyr7
\font\cyrt=wncyr5

\nc{\ann}{\mrm{ann}}
\nc{\Aut}{\mrm{Aut}}
\nc{\can}{\mrm{can}}
\nc{\Cont}{\mrm{Cont}}
\nc{\rchar}{\mrm{char}}
\nc{\cok}{\mrm{coker}}
\nc{\dtf}{{R-{\rm tf}}}
\nc{\dtor}{{R-{\rm tor}}}

\nc{\Div}{{\mrm Div}}
\nc{\End}{\mrm{End}}
\nc{\Ext}{\mrm{Ext}}
\nc{\Fil}{\mrm{Fil}}
\nc{\Fr}{\mrm{Fr}}
\nc{\Frob}{\mrm{Frob}}
\nc{\Gal}{\mrm{Gal}}
\nc{\GL}{\mrm{GL}}
\nc{\Hom}{\mrm{Hom}}
\nc{\hsr}{\mrm{H}}
\nc{\hpol}{\mrm{HP}}
\nc{\id}{\mrm{id}}
\nc{\im}{\mrm{im}}
\nc{\incl}{\mrm{incl}}
\nc{\length}{\mrm{length}}
\nc{\mchar}{\rm char}
\nc{\mpart}{\mrm{part}}
\nc{\ql}{{\QQ_\ell}}
\nc{\qp}{{\QQ_p}}
\nc{\rank}{\mrm{rank}}
\nc{\rcot}{\mrm{cot}}
\nc{\rdef}{\mrm{def}}
\nc{\rdiv}{{\rm div}}
\nc{\rtf}{{\rm tf}}
\nc{\rtor}{{\rm tor}}
\nc{\res}{\mrm{res}}
\nc{\SL}{\mrm{SL}}
\nc{\Spec}{\mrm{Spec}}
\nc{\tor}{\mrm{tor}}
\nc{\Tr}{\mrm{Tr}}
\nc{\tr}{\mrm{tr}}

\nc{\ab}{\mathbf{Ab}}
\nc{\DRB}{\mathbf{DRB}}
\nc{\OPA}{\mathbf{OPA}}
\nc{\Alg}{{\mathbf{Alg}}}
\nc{\ALG}{{\mathbf{ALG}}}
\nc{\RB}{\mathbf{RB}}
\nc{\RBA}{\mathbf{RBA}}
\nc{\bfk}{{\bf k}}
\nc{\bfone}{{\bf 1}}
\nc{\bfzero}{{\bf 0}}
\nc{\detail}{\marginpar{\bf More detail}
    \noindent{\bf Need more detail!}
    \svp}
\nc{\Diff}{\mathbf{Diff}}
\nc{\gap}{\marginpar{\bf Incomplete}\noindent{\bf Incomplete!!}
    \svp}
\nc{\FMod}{\mathbf{FMod}}
\nc{\Int}{\mathbf{Int}}
\nc{\Mon}{\mathbf{Mon}}
\nc{\Mult}{{\mathbf{Mlt}}}
\nc{\remarks}{\noindent{\bf Remarks: }}
\nc{\Rep}{\mathbf{Rep}}
\nc{\Rings}{\mathbf{Rings}}
\nc{\Sets}{\mathbf{Sets}}

\nc{\lir}[1]{\textcolor{purple}{\underline{Li:}#1 }}

\nc{\BA}{{\mathbb A}}
\nc{\CC}{{\mathbb C}}
\nc{\DD}{{\mathbb D}}
\nc{\EE}{{\mathbb E}}
\nc{\FF}{{\mathbb F}}
\nc{\GG}{{\mathbb G}}
\nc{\HH}{{\mathbb H}}
\nc{\LL}{{\mathbb L}}
\nc{\NN}{{\mathbb N}}
\nc{\PP}{{\mathbb P}}
\nc{\QQ}{{\mathbb Q}}
\nc{\RR}{{\mathbb R}}
\nc{\TT}{{\mathbb T}}
\nc{\VV}{{\mathbb V}}
\nc{\ZZ}{{\mathbb Z}}
\nc{\TP}{\widetilde{P}}
\nc{\lp}{\widehat{P}}
\nc{\lpt}{\widehat{P}^{\, \omega}}
\nc{\lqb}{\widehat{Q}}
\nc{\lqt}{\widehat{Q}^{\, \omega}}
\nc{\ld}{\hat{d}}
\nc{\ldt}{\hat{d}^{\, \omega}}
\nc{\lqst}{\hat{q}^{\, \omega}}
\nc{\lqs}{\hat{q}}
\nc{\lpp}{\overline{P'}}
\nc{\dee}{\mathrm{Deg}}
\nc{\Mpf}{\mathbf{M}}
\nc{\m}{\iota}

\nc{\cala}{{\mathcal A}}
\nc{\calc}{{\mathcal C}}
\nc{\cald}{\mathcal{D}}
\nc{\cale}{{\mathcal E}}
\nc{\calf}{{\mathcal F}}
\nc{\calg}{{\mathcal G}}
\nc{\calh}{{\mathcal H}}
\nc{\cali}{{\mathcal I}}
\nc{\call}{{\mathcal L}}
\nc{\calm}{{\mathcal M}}
\nc{\caln}{{\mathcal N}}
\nc{\calo}{{\mathcal O}}
\nc{\calp}{{\mathcal P}}
\nc{\calr}{{\mathcal R}}
\nc{\cals}{{\mathcal S}}
\nc{\calt}{{\Omega}}
\nc{\calw}{{\mathcal W}}
\nc{\calx}{{\mathcal X}}
\nc{\CA}{\mathcal{A}}

\nc{\fraka}{{\mathfrak a}}
\nc{\frakb}{\mathfrak{b}}
\nc{\frakB}{{\frak B}} \nc{\frakm}{{\frak
m}} \nc{\frakM}{{\frak M}}
\nc{\frakp}{{\frak p}}
\nc{\frakS}{{\frak S}}
\nc{\frakA}{{\frak A}} \nc{\frakx}{{\frakx}}

\nc{\Dif}{\mathbf{Dif}}
\nc{\DIF}{\mathbf{DIF}}
\nc{\ADR}{\mathbf{ADR}}
\nc{\OA}{\mathbf{OA}}
\nc{\ODA}{\mathbf{ODA}}
\nc{\ORB}{\mathbf{ORB}}
\nc{\DaRB}{\mathbf{DaRB}}
\nc{\G}{\mathbf{G}}
\nc{\C}{\mathbf{C}}
\nc{\A}{\mathbf{A}}
\nc{\B}{\mathbf{B}}
\nc{\T}{\mathbf{T}}

\begin{document}
\title[Classification of extensions, liftings and distributive laws]{Classification of operator extensions, monad liftings and distributive laws for differential algebras and Rota-Baxter algebras}

\date{\today}

\author{Shilong Zhang}
\address{College of Science, Northwest A\&F University, Yangling 712100, Shaanxi, China, and
Department of Mathematics, Lanzhou University, Lanzhou, Gansu, 730000, China}
\email{2663067567@qq.com}

\author{Li Guo}
\address{
Department of Mathematics and Computer Science,
Rutgers University,
Newark, NJ 07102, USA}
\email{liguo@newark.rutgers.edu}

\author{William Keigher}
\address{
Department of Mathematics and Computer Science,
Rutgers University,
Newark, NJ 07102, USA}
\email{keigher@newark.rutgers.edu}

\begin{abstract}
Generalizing the algebraic formulation of the First Fundamental Theorem of Calculus (FFTC), a class of constraints involving a pair of operators was considered in~\cite{ZGK2}. For a given constraint, the existences of extensions of differential and Rota-Baxter operators, of liftings of monads and comonads, and of mixed distributive laws are shown to be equivalent. In this paper, we give a classification of the constraints satisfying these equivalent conditions.
\end{abstract}

\subjclass[2010]{18C15, 13N99, 16W99}

\keywords{
Rota-Baxter algebra, differential algebra, cover of operators, extension of operators, monad, distributive law}

\maketitle

\vspace{-1cm}

\tableofcontents

\vspace{-1cm}

\setcounter{section}{0}

\allowdisplaybreaks

\section{Introduction}

The algebraic study of analysis has a long history. In the 1930s, the notion of a {\bf differential ring} or {\bf algebra} was introduced by Ritt~\cite{Ri} to give an algebraic study of differential analysis and differential equations. Here a differential algebra is an (associative) algebra $R$ with a linear operator $d$ satisfying the Leibniz rule
\begin{equation}
 d(xy)=d(x)y+xd(y)\ \text{ for all }\ x,\ y\in R.
\mlabel{eq:der10}
\end{equation}
Through the later work of Kolchin and many other mathematicians, differential algebra has been developed into a vast area including differential Galois groups, differential algebraic groups and differential algebraic geometry, with broad applications in number theory, logic and mechanical proof of mathematical theorems~\cite{Kol,SP,Wu2}.

The algebraic abstraction of the integral analysis came much later, as a byproduct of the work of G. Baxter in probability in 1960~\cite{Ba}. A Baxter algebra, later called {\bf Rota-Baxter algebra}, is an algebra $R$ with a linear operator $P$ such that
\begin{equation}
 P(x)P(y)=P(P(x)y)+P(xP(y))+\lambda P(xy)\ \text{ for all }\ x,\ y\in R.
\mlabel{eq:bax1}
\end{equation}
Here $\lambda$ is a given scalar in the base ring, called the {\bf weight} of the Rota-Baxter operator. After the pioneering work of Cartier and Rota~\cite{Ca,Ro} in combinatorics, the recent developments of Rota-Baxter algebras have ranged from multiple zeta values in number theory to renormalization of perturbation quantum field theory~\cite{Bai,CK,Gub,GK1,GZ,RR,Ro}.

As a differential analog of a Rota-Baxter operator of weight $\lambda$, a {\bf differential operator of weight $\lambda$}~\cite{GK3} is defined to satisfy the equation
\begin{equation}
 d(xy)=d(x)y+xd(y)+\lambda d(x)d(y)\ \text{ for all }\ x,\ y\in R
\mlabel{eq:der1}
\end{equation}
and $d(\bfone_{R})=0.$

With the algebraizations of both differential and integral analyses in place, it is natural to formulate an algebraic abstraction of the two analyses through the well-known First Fundamental Theorem of Calculus (FFTC), leading to the notion of a differential Rota-Baxter algebra with weight.
To be precise, a differential Rota-Baxter algebra of weight $\lambda$ is a triple $(R,d,P)$ consisting of
\begin{enumerate}
\item
an algebra $R$,
\item
a differential operator $d$ of weight $\lambda$ on $R$, and
\item
a Rota-Baxter operator $P$ of weight $\lambda$ on $R$,
\end{enumerate}
such that
\begin{equation}
d P = \id_R,
\mlabel{eq:fftc}
\end{equation}
reflecting the FFTC. See~\mcite{GGR,GRR,RR} for a variation, called an integro-differential algebra.

This natural algebraic abstraction of the FFTC has quite remarkable categorical implications in terms of liftings of monads and mixed distributive laws, as shown in~\cite{ZGK} which has attracted interests from combinatorics, differential algebra, probability and computer science~\cite{CL,Fr,Ja,St}.

We fix an algebra $R$ and let $R^\NN$ denote the Hurwitz series algebra over $R$~\cite{Ke,GK3}. Then $R^\NN$, with a natural differential operator $\partial_R$, is the cofree differential algebra on $R$. Further, we have a comonad, denoted by $\C$, giving differential algebras~\cite{ZGK}. Also let $\sha(R)$ be the mixable shuffle product algebra~\cite{GK1}. Then $(\sha(R), P_R)$, where $P_R$ is a naturally defined Rota-Baxter operator, is the free Rota-Baxter algebra on $R$ which results in a monad, denoted by $\T$, giving Rota-Baxter algebras~\cite{ZGK}.
In~\mcite{GK3}, a differential operator on $R$ is uniquely extended to $\sha(R)$, enriching $\sha(R)$ to be the free differential Rota-Baxter algebra and giving a lifting of the monad $\T$. Further, by the lifting, we obtain a mixed distributive law of the monad $\T$ over the comonad $\C$~\cite{ZGK}. Similarly, given a Rota-Baxter operator on $R$, we construct a cover\footnote{The term ``coextension" in~\cite{ZGK2} is replaced by ``cover" in this paper.} of the operator on $R^\NN$, and enrich $R^\NN$ to be the cofree differential Rota-Baxter algebra which gives a lifting of the comonad $\C$~\cite{ZGK}. The same mixed distributive law also follows.

These results show that the coupling of a differential operator and a Rota-Baxter operator via FFTC leads to the existences of extensions of differential operators, of covers of Rota-Baxter operators, of liftings of monads, and of mixed distributive laws. Then there are the following natural problems.

\begin{problem}
How are the categorical properties of extensions of differential operators, covers of Rota-Baxter operators, liftings of monads, and  mixed distributive laws interrelated to one another?
 \label{pr:one}
\end{problem}

\begin{problem}
Are these equivalent categorical properties unique to this coupling via FFTC?  In other words, are there other examples where these equivalent properties hold?
\label{pr:two}
\end{problem}

To investigate Problem~\ref{pr:one}, a follow-up study of~\cite{ZGK}
was carried out in~\cite{ZGK2} where the identity in the FFTC is viewed as an example
of a polynomial identity in two noncommutative variables symbolizing the differential operator and Rota-Baxter operator. Thus we work in the noncommutative polynomial algebra $\bfk\langle x,y\rangle$ in two variables $x, y$ and regard each polynomial $\omega=\omega(x,y)$ in $\bfk\langle x,y\rangle$ as a constraint between two linear operators $q$ and $Q$, both defined on an algebra $R$, given by a formal identity $\omega(q, Q)=0$. When $q$ and $Q$ are the differential operator and Rota-Baxter operator respectively, $\omega(x,y)=xy-1$
gives the FFTC as in ~\eqref{eq:fftc}.

Before attempting the most general case, we consider a class $\calt\subset \bfk\langle x,y\rangle$ of constraints in which to investigate Problem~\ref{pr:one}.
The class of constraints is
\begin{equation}
\calt:=xy+\bfk[x]+y\bfk[x] = \left\{
xy-(\phi(x) +y\psi(x))\,|\, \phi,\psi\in\bfk[x]\right\}.
\mlabel{eq:Ti}
\end{equation}

It was a pleasant surprise to find out in~\cite{ZGK2} that, for these constrains, the following categorical properties are in fact equivalent: extensions of differential operators, covers of Rota-Baxter operators, liftings of monads, and existence of mixed distributive laws. See Theorem~\ref{thm:main} for a precise statement. The theorem provides an answer to Problem~\ref{pr:one}.

The purpose of this paper is to address Problem~\ref{pr:two}, namely the dependence of these categorical properties on the constraints. According to~\cite{ZGK}, the polynomial $xy-1$ representing the FFTC provides an example where all these equivalent categorical properties are fulfilled. As these categorical properties are quite strong, one would expect that polynomials with such properties are quite rare. In this paper we confirm this expectation, by explicitly displaying all such polynomials, through their connection with the covers of Rota-Baxter operators.

Here is an outline of the paper. In Section~\ref{sec:coextension}, we give the preliminary concepts and results leading to Theorem~\mref{thm:main} on the equivalence among the existences of covers of operators, of extensions of operators, of liftings of monads, and of mixed distributive laws. We then state Theorem~\ref{thm:lr0} classifying all polynomials in $\calt$ satisfying these equivalent properties. In Section~\ref{sec:proof}, Theorem~\ref{thm:lr0} is rephrased as Theorem~\mref{thm:lr} and proved in several steps. It would be interesting to determine whether this equivalence of categorical properties holds for more general  polynomials in $\bfk\langle x,y\rangle$ and to achieve a classification of such polynomials. Conjecture~\ref{conj:omega} gives a formulation in this direction.

\smallskip

Throughout the paper, we fix a commutative ring $\bfk$ with identity and an element $\lambda \in \bfk$. Unless otherwise noted, we work in the categories of commutative $\bfk$-algebras with identity, with or without linear operators. All operators and tensor products are also taken over $\bfk$. Thus references to $\bfk$ will be suppressed unless doing so can cause confusion.
We let $\NN$ denote the additive monoid of
natural numbers $\{0,1,2,\ldots\}$ and
$\NN_+=\{ n\in \NN\mid n>0\}$ the positive integers.
Let $\delta_{i,j}, i, j\in\NN$ denote the Kronecker delta.
For categorical notations, we follow~\cite{Ma}.

\section{Background and the statement of the main theorem}
\mlabel{sec:coextension}

In this section we provide background to state the main theorem and prove preliminary results required in the proof of the main theorem.
In Section~\ref{ss:freecofree}, we review free Rota-Baxter algebras and cofree differential algebras, and the corresponding adjoint functor pairs. We also prove a property in a special case which will be applied repeatedly in the proof of the main theorem. In Section~\ref{ss:extn}, we start with the category of operated algebras and consider its enrichments by adding differential operators or Rota-Baxter operators. We give covers of operators or extensions of operators in a operated algebra to certain objects in these enriched categories. Building on these preparations, we state in Section~\ref{ss:main} the theorem on the equivalence of extensions or covers of operators, liftings of monads and existence of distributive laws, and the main theorem which gives a classification of the constraints for which each and hence all the equivalent conditions hold.

Additional details can be found in~\mcite{Gub,GK3,ZGK2}.
The proof of the main theorem will be given in the next section.

\subsection{Free Rota-Baxter algebras and cofree differential algebras}
\label{ss:freecofree}

Recall from~\cite{Gub,GK1} the construction of the free commutative Rota-Baxter algebra $\sha(A)$ of weight $\lambda$ on a commutative
algebra $A$ with identity $\bfone_A$. As a module, we have
$$\sha(A) = \bigoplus\limits_{i\in\NN_+}A^{\otimes i}=A\oplus (A\otimes A)\oplus (A\otimes A\otimes A)\oplus\cdots.$$

Define an operator $P_A$ on
$\sha(A)$ by assigning
\[ P_A( x_0\otimes x_1\otimes \cdots \otimes x_n):
=\bfone_A\otimes x_0\otimes x_1\otimes \cdots\otimes x_n \]
for all
$x_0\otimes x_1\otimes \cdots\otimes x_n\in A^{\otimes (n+1)}$
and extending by additivity.

Then by~\cite[Theorem~4.1]{GK1}, the module $\sha(A)$, with the mixable shuffle product, the operator $P_A$ and the natural embedding
$j_A:A\rightarrow \sha(A)$,
is a free Rota-Baxter algebra of weight $\lambda$ on $A$.
More precisely,
for any Rota-Baxter algebra $(R,P)$ of weight $\lambda$ and any
algebra homomorphism
$\varphi:A\rar R$, there exists
a unique Rota-Baxter algebra homomorphism
$\tilde{\varphi}:(\sha(A),P_A)\rar (R,P)$ such that
$\varphi = \tilde{\varphi} j_A$.

For later use, we give a class of Rota-Baxter algebras with non-zero Rota-Baxter operators.
\begin{exam}
Take $A = \bfk$ in $\sha(A)$. Then~~\cite[Proposition~6.1]{GK1} states that $\sha(\bfk)$ is an algebra with basis $z_i:=1^{\ot (i+1)}\in \bfk^{\ot (i+1)}$ for each $i\in \NN$.
The multiplication on $\sha(\bfk)$ is given by
\begin{equation}
z_m  z_n=\sum\limits_{j=0}^{m}{m+n-j\choose n}{n\choose j}\lambda^j z_{m+n-j}\quad  \text{for all}\  m, n\in\NN.
\mlabel{eq:tmng}
\end{equation}
In particular, when $\lambda=0$, one sees
\begin{equation}
z_m z_n={m+n\choose n}\,z_{m+n},
\mlabel{eq:tmn}
\end{equation}
giving the divided power algebra.

The identity element of $\sha(\bfk)$ is $z_0$ and the operator $P_{\bfk}:\sha(\bfk)\rightarrow \sha(\bfk)$ is given by
$$P_{\bfk}(z_i)=z_{i+1}\quad \text{for each}\ i\in \NN.$$

For a given $m\in \NN_+$, $I_m:=\oplus_{i\geq m}\bfk z_{i}$ is a Rota-Baxter ideal of $(\sha(\bfk), P_{\bfk})$, that is, $$\sha(\bfk)I_m\subseteq I_m, \quad P_{\bfk}(I_m)\subseteq I_m,$$
giving rise to the quotient Rota-Baxter algebra $(\sha(\bfk)/I_m, \overline{P_{\bfk}})$.
Then for $m\geq 2$,
$$\overline{P_{\bfk}}(\overline{z_{m-2}})=\overline{z_{m-1}}\neq \overline{0}, \quad \overline{P_{\bfk}}(\overline{z_{m-1}}) =\overline{z_m}=\overline{0}\in\sha(\bfk)/I_m.$$
\mlabel{ex:couex}
\end{exam}

We let $\ALG$ denote the category of commutative algebras, and let $\RBA_{\lambda}$, or simply $\RBA$, denote the category
of commutative Rota-Baxter algebras of weight $\lambda$. Then we have a functor $F:\ALG \rar \RBA$ given on
objects $A$ by $F(A) = (\sha(A),P_A)$ and on morphisms
$\varphi:A \rar B$ by $$F(\varphi)\left(\sum^{k}_{i=1}a_{i0}\otimes a_{i1}\otimes\cdots \otimes a_{in_{i}}\right)=\sum^{k}_{i=1}\varphi(a_{i0})\otimes \varphi(a_{i1})\otimes\cdots \otimes \varphi(a_{in_{i}})$$
for any $\sum\limits^{k}_{i=1}a_{i0}\otimes a_{i1}\otimes\cdots \otimes a_{in_{i}}\in\sha(A)$.

Then the universal property of the free Rota-Baxter algebras $\sha(A)$ has the following formulation.

\begin{prop}$($\cite[Corollary~2.4]{ZGK}$)$
The functor $F:\ALG \rightarrow \RBA$ defined above is the left adjoint of the forgetful functor
$U:\RBA \rightarrow \ALG$.
\mlabel{pp:adjrba}
\end{prop}

Furthermore, the adjunction in Proposition~\ref{pp:adjrba} provides a monad $\mathbf{T}=\mathbf{T}_{\RBA}=\langle T,\eta,\mu\rangle$ on $\ALG$ giving Rota-Baxter algebras~\cite[\S~2.2]{ZGK}. Indeed, for every algebra $A$, $T(A)=\sha(A)$,
$\eta_{A}$ is the natural embedding from $A$ to $\sha(A)$, and $\mu_{A}:\sha(\sha(A))\rightarrow \sha(A)$ is extended additively from
 \begin{eqnarray*}
&&\mu_{A}( (a_{00}\otimes\cdots\otimes a_{0n_{0}})\otimes\cdots\otimes(a_{k0}\otimes\cdots\otimes a_{kn_{k}}))\\
&=&(a_{00}\otimes\cdots\otimes a_{0n_{0}})  P_{A}(\cdots P_{A}(a_{k0}\otimes\cdots\otimes a_{kn_{k}})\cdots),
\end{eqnarray*}
where
$$(a_{00}\otimes\cdots\otimes a_{0n_{0}})\otimes\cdots\otimes(a_{k0}\otimes\cdots\otimes a_{kn_{k}})\in\sha(\sha(A)) \,\,\text{with}\,\, a_{i0}\otimes\cdots\otimes a_{in_{i}}\in A^{\ot (n_{i}+1)}$$
for $n_0,\ldots,n_k\geq 0$ and $0\leq i\leq k$.

\smallskip

Next we review some background on differential algebras with weights, defined in ~(\mref{eq:der1}), and refer the reader to~\mcite{GK3} for details.

For any algebra $A$, let $A^{\NN}$ denote the $\bfk$-module of functions $f:\NN \rightarrow A$. We also view $f \in A^{\NN}$ as a sequence
$$(f_n)=(f_0,f_1,\cdots) \, \text{with} \, f_n:= f(n) \in A.$$
Following~\cite[\S~2.3]{GK3}, the {\bf $\lambda$-Hurwitz product} on $A^{\NN}$ is given by
\begin{equation}
(fg)_n = \sum_{k=0}^{n}\sum_{j=0}^{n-k} {n\choose k} {n-k\choose j}
\lambda^{k}f_{n-j}g_{k+j}\quad \text{for all}\  f, g\in R^\NN.
\mlabel{eq:hurprod}
\end{equation}
In the special case when $\lambda =0$, we have
\begin{equation}
(fg)_n =\sum_{j=0}^{n} {n\choose j} f_{n-j}g_{j}.
\mlabel{eq:hurprod0}
\end{equation}
With this product, $A^{\NN}$ is called the {\bf algebra of $\lambda$-Hurwitz series} over $A$.
Further define
\begin{equation*}
\partial_A:A^{\NN} \rightarrow A^{\NN}, \quad \partial_A(f)_n = f_{n+1}\quad\text{ for all }\  f \in A^{\NN}, n\in \NN.
\mlabel{eq:partial}
\end{equation*}
Then $\partial_A$ is a differential operator of weight $\lambda$ on $A^{\NN}$, making $(A^{\NN},\partial_A)$ into a  differential algebra of weight $\lambda$.
We also obtain a recursive formula for $(fg)_n$:
\begin{equation}
(fg)_{n+1}=(\partial_A(fg))_n =(\partial_A(f)g)_n+(f\partial_A(g))_n+
(\lambda\partial_A(f)\partial_A(g))_n \quad\text{ for all }\  f, g \in A^{\NN}, n\in \NN.
\label{eq:recHur}
\end{equation}

Let $\DIF$ denote the category of differential algebras of weight $\lambda$, and $G:\ALG \rightarrow \DIF$ be a functor given on objects $A$ by $G(A): = (A^{\NN}, \partial_A)$ and on morphisms
$\varphi:A \rightarrow B$ by
$$G(\varphi):=\varphi^{\NN}:(A^{\NN}, \partial_A)\rar (B^{\NN}, \partial_B),\quad (\varphi^{\NN}(f))_n=\varphi(f_n),\, f\in A^\NN, n\in \NN.$$

\begin{prop}$($\text{See}~\cite[Proposition~2.8]{GK3}$)$
The functor $G:\ALG \rightarrow \DIF$ is the right adjoint of the forgetful functor $V:\DIF \rightarrow \ALG$. In other words, the differential algebra $(A^{\NN},\partial_A)$, together with the algebra homomorphism
\begin{equation}
\varepsilon_{A}: A^{\NN}\rightarrow A, \quad \varepsilon_A(f): = f_0\quad\text{for all}\  f \in A^{\NN},
\mlabel{eq:vare}
\end{equation}
is a cofree differential algebra of weight $\lambda$ on the algebra $A$.
\mlabel{prop:cofree}
\end{prop}

The adjunction in Proposition~\ref{prop:cofree} gives rise to a comonad $\C = \langle C,\varepsilon,\delta \rangle$ on $\ALG$ giving differential algebra~\cite[\S~3]{ZGK}. For every $A\in \ALG$, $C(A)=A^{\NN}$, $\varepsilon_A:A^{\NN} \rightarrow A$ is a surjection with $\varepsilon_A(f)= f_0$, and $\delta_A:A^{\NN} \rightarrow (A^{\NN})^{\NN}$ is defined by
$$(\delta_A(f)_m)_n = f_{m+n}, \quad f \in A^{\NN}, m, n \in \NN.$$
Fix a comonad $\tilde{\C}= \langle \tilde{C},\tilde{\varepsilon},\tilde{\delta} \rangle$ on $\RBA$. If the following identities hold
$$U\tilde{C}=CU,\quad U\tilde{\varepsilon}=\varepsilon U \,\,\text{and}\,\, U\tilde{\delta}=\delta U,$$
then $\tilde{\C}$ is said to lift the comonad $\mathbf{C}$.
Dually, a monad $\tilde{\T}= \langle \tilde{T},\tilde{\eta},\tilde{\mu} \rangle$ on $\DIF$ lifting the monad $\mathbf{T}$ from Proposition~\ref{pp:adjrba} satisfies
$$V\tilde{T}=TV,\quad V\tilde{\eta}=\eta V \,\,\text{and}\,\, V\tilde{\mu}=\mu V.$$ The monad $\T$ , comonad $\C$ and their lifting forms will be used in the formulation of Theorem~\ref{thm:main}.

\begin{exam}\label{ex:coud}{\rm
On the Rota-Baxter algebra $(\sha(\bfk), P_\bfk)$ in Example~\mref{ex:couex}, define
$$d: \sha(\bfk)\to \sha(\bfk), \quad d(z_0)=0, d(z_n)=z_{n-1}\quad \text{for all} \  n\in \NN_+.$$
Then $(\sha(\bfk),d)$ is a differential algebra of weight $\lambda$. By \cite[Corollary~3.7]{GK2}, the completion of $(\sha(\bfk),d)$ is isomorphic to the algebra $\bfk^\NN$ of Hurwitz series over $\bfk$.
}
\end{exam}

\subsection{Covers and extensions of operators} \label{ss:extn}

In~\mcite{ZGK2}, we also introduced a class of polynomials that serve as relations between a pair of operators as follows.
In the noncommutative polynomial algebra $\bfk\langle x, y\rangle$ in two variables $x$ and $y$, consider the subset
\begin{equation}
\calt:=xy+\bfk[x]+y\bfk[x] = \{
xy-(\phi(x) +y\psi(x))\,|\, \phi,\psi\in\bfk[x]\}.
\mlabel{eq:T}
\end{equation}
Let $q$ and $Q$ be two operators on an algebra $R$. Each $\omega:=\omega(x, y)\in\calt$ is regarded as a relation $\omega(q, Q)=0$ between $q$ and $Q$. As a special case, $\omega=xy-1$ is regarded as the relation $\omega(d, P)=dP-\id_R =0$ between the operators $d$ and $P$ in a differential Rota-Baxter algebra $(R, d, P)$~\cite{GK3} reflecting the First Fundamental Theorem of Calculus.

\begin{defn}
{\rm
An {\bf operated algebra}~\cite{Gop,Ku} is an algebra $R$ with a linear operator $Q$ on $R$, thus denoted as a pair $(R, Q)$. Let $\OA$ denote the category of operated algebras.
}
\end{defn}

As its enrichment, we have

\begin{defn}$($\cite[Definition~2.6]{ZGK2}$)$
{\rm
For a given $\omega\in\calt$ and $\lambda\in \bfk$, we say that the triple $(R,d,Q)$ is a
{\bf type $\omega$ operated differential algebra of weight $\lambda$} if
\begin{enumerate}
\item $(R,d)$ is a differential algebra of weight $\lambda$,
\item $(R,Q)$ is an operated algebra, and
\item $\omega(d, Q)=0$, that is,
   \begin{equation}
   dQ=\phi(d)+Q\psi(d).
   \mlabel{eq:dP}
   \end{equation}
\end{enumerate}
}
\mlabel{de:oda}
\end{defn}

There is a one-to-one correspondence between operators $\calp$ on $R^\NN$ and sequences $(\calp_n)$ of linear maps where, for each $n\in\NN$, $\calp_n:R^\NN\rightarrow R$ is given by
\begin{equation*}
\calp_n(f):=\calp(f)_n\quad\text{for all}\  f\in R^\NN.
\mlabel{eq:coexten}
\end{equation*}
For any operators $\mathcal{Q}, \mathcal{J}$ on $R^\NN$, and each $f\in R^\NN, n\in\NN$, we obtain
\begin{equation}
(\partial_R\mathcal{Q})_n (f)=(\partial_R(\mathcal{Q}(f)))_n=\mathcal{Q}_{n+1}(f),\quad (\mathcal{Q}\mathcal{J})_n(f)=(\mathcal{Q}(\mathcal{J}(f)))_n =\mathcal{Q}_n(\mathcal{J}(f)).
\mlabel{eq:pf0}
\end{equation}

We now recall the notion of a cover (called coextension in~\mcite{ZGK2}) of an operator on an algebra to the cofree differential algebra generated by this algebra.

\begin{defn}
For a given operator $Q:R\to R$, we call an operator $\lqb:R^\NN\to R^\NN$ a {\bf cover of $Q$ on $R^\NN$} if for all $f\in R^\NN$, we have $\lqb_0(f)=Q(f_0)$. That is, the following diagram commutes
\[\xymatrix{
R^\NN\ar[d]_{\varepsilon_R} \ar[r]^{\lqb}&R^\NN\ar[d]^{\varepsilon_R}\\
R \ar[r]_{Q}&R
}\]
where $\varepsilon_R$ is given in ~$($\mref{eq:vare}$)$.
\end{defn}

The operator $\lqb$ is called a cover of $Q$ because $\vep_R$ is surjective. For each $\omega\in\calt$, we established the existence and uniqueness of a cover as the following proposition shows.

\begin{prop}$($\cite[Proposition~2.6]{ZGK2}$)$
Let $Q$ be an operator on an algebra $R$. For a given $\omega=xy-(\phi(x)+y\psi(x))\in \calt$ with $\phi, \psi\in \bfk[x]$, $Q$ has a unique cover $\lqt:(R^\NN, \partial_R)\to (R^\NN, \partial_R)$ such that $\omega(\partial_R, \lqt) = 0$, that is$:$
\begin{equation}
\partial_R \lqt=\phi(\partial_R) +\lqt\psi(\partial_R).
\mlabel{eq:lpproperty}
\end{equation}
Thus the triple $(R^\NN, \partial_R, \lqt)$ is a type $\omega$ operated differential algebra.
\mlabel{prop:extrb}
\end{prop}
For a given $\omega\in\calt$, let $\ODA_\omega$ denote the category of type $\omega$ operated differential algebras of weight $\lambda$ in Definition~\mref{de:oda}. Thanks to Proposition~\mref{prop:extrb}, we obtain a functor
\begin{equation}
 G_\omega: \OA \to \ODA_\omega.
 \mlabel{eq:coefun}
\end{equation}

Applying ~(\mref{eq:pf0}), ~(\mref{eq:lpproperty}) is equivalent to
\begin{equation}
\lqt_{n+1}=\phi(\partial_R)_n +\lqt_n\psi(\partial_R)\quad\text{for all}\  n\in\NN.
\mlabel{eq:lptn}
\end{equation}

The following equivalent characterizations of a Rota-Baxter algebra in terms of covers will be useful in the proof of our main result Theorem~\ref{thm:lr0}.

\begin{prop}
Let $(R, Q)$ be an operated algebra and $\lqb$ be any cover of $Q$ to $R^\NN$.
\begin{enumerate}
\item
$(R, Q)$ is a Rota-Baxter algebra of weight $\lambda$ if and only if
\begin{equation}
\lqb_0(f)\lqb_0(g)=\lqb_0(\lqb(f)g)+
\lqb_0(f\lqb(g))+\lambda \lqb_0(fg) \quad \text{for all} \  f, g \in R^\NN.
\mlabel{eq:q0}
\end{equation}
\mlabel{it:pq0}
\item
$(R^\NN, \lqb)$ is a Rota-Baxter algebra of weight $\lambda$ if and only if for all $f, g \in R^\NN$, $n\in\NN$,
\begin{equation}
\sum_{k=0}^{n}\sum_{j=0}^{n-k} {n\choose k} {n-k\choose j}
\lambda^{k}\lqb_{n-j}(f)\lqb_{k+j}(g)=\lqb_n(\lqb(f)g)+
\lqb_n(f\lqb(g))+\lambda \lqb_n(fg).
\mlabel{eq:qn}
\end{equation}
\mlabel{it:pqn}
\end{enumerate}
\mlabel{prop:qq}
\end{prop}
The proof of Proposition~\mref{prop:qq} is straightforward.

The following special case of Proposition~\mref{prop:qq}.(\mref{it:pqn}) will be used repeatedly in the proof of Theorem~\mref{thm:lr}.(\mref{it:weight0}). When $\lambda=0$, ~(\mref{eq:qn}) becomes
\begin{equation*}
\sum_{j=0}^{n} {n\choose j}
\lqb_{n-j}(f)\lqb_{j}(g)=\lqb_n(\lqb(f)g)+
\lqb_n(f\lqb(g)) \quad\text{for all }\ f, g \in R^\NN, n\in\NN.
\mlabel{eq:qn0}
\end{equation*}
As a consequence, we have

\begin{coro}
Let $(R,P)$ be a Rota-Baxter algebra of weight $0$. If there are $f, g\in R^\NN$ such that
\begin{equation}
\left(\lpt_1(\lpt(f)g)+\lpt_1(f\lpt(g))\right)-
\left(\lpt_0(f)\lpt_1(g)+\lpt_1(f)\lpt_0(g)\right)
\neq 0,
\mlabel{eq:qn01}
\end{equation}
then the cover $\lpt$ of $P$ to $(R^\NN, \partial_R)$ is not a Rota-Baxter operator of weight $0$.
\mlabel{co:qq0}
\end{coro}

Next we recall the definition of extensions of operators in~\cite{ZGK2}.
\begin{defn}
For a given operator $q:R\to R$ satisfying $q(\bfone_R)=0$, we call an operator $\lqs:\sha(R)\to \sha(R)$ an {\bf extension} of $q$ to $\sha(R)$ if $\lqs|_R=q$. That is, the following diagram commutes
\[\xymatrix{
R\ar[d]_{j_R} \ar[r]^{q}&R\ar[d]^{j_R}\\
\sha(R) \ar[r]_{\lqs}&\sha(R)
}\]
where $j_R:R\rightarrow \sha(R)$ is a natural embedding.
\end{defn}

Since $\bfone_{\ssha(R)}=\bfone_R$ by the definition of the mixable shuffle product on $\sha(R)$, we have $\lqs(\bfone_{\ssha(R)})=q(\bfone_R)=0$ for an extension $\lqs$ of $q$ to $\sha(R)$.

\begin{defn}$($\cite[Definition~2.9]{ZGK2}$)$
For a given $\omega\in\calt$ and $\lambda\in \bfk$, we say that the triple $(R,q,P)$ is a
{\bf type $\omega$ operated Rota-Baxter algebra of weight $\lambda$} if
\begin{enumerate}
\item $(R,q)$ is an operated algebra with the property $q(\bfone_R)=0$,
\item $(R,P)$ is a Rota-Baxter algebra of weight $\lambda$, and
\item $\omega(q, P)=0$, that is,
   \begin{equation*}
   qP=\phi(q)+P\psi(q).
   \mlabel{eq:qP}
   \end{equation*}
\end{enumerate}
\end{defn}

\begin{prop}$($\cite[Proposition~2.11]{ZGK2}$)$
Let $(R, q)$ be an operated algebra where the operator $q$ satisfies $q(\bfone_R)=0$. For a given $\omega=xy-(\phi(x)+y\psi(x))\in \calt$ with $\phi, \psi\in\bfk[x]$, $q$ has a unique extension $\lqst: (\sha(R), P_R)\rightarrow (\sha(R), P_R)$ with the following property: for $u=u_0\ot u'\in R^{\ot (n+1)}$ with $u'\in R^{\ot n}$,
\begin{equation}
\lqst(u)=q(u_0)\ot u'+(u_0+\lambda q(u_0))(\phi(\lqst)+ P_R\psi(\lqst))(u')
\mlabel{eq:1leibdform}
\end{equation}
and
\begin{equation}
\lqst(\oplus_{i=1}^{n}R^{\ot i})\subseteq\oplus_{i=1}^{n}R^{\ot i}\quad\text{for each}\  n\in\NN_+.
\mlabel{eq:qcon}
\end{equation}
The triple
$(\sha(R), \lqst, P_R)$ is a type $\omega$ operated Rota-Baxter algebra.
\mlabel{prop:extd}
\end{prop}
We let $\OA_0$ denote the category of operated algebras $(R, q)$ with the property $q(\bfone_R)=0$. Thus $\DIF$ is a subcategory of $\OA_0$. Let $\ORB_\omega$ denote the category of type $\omega$ operated Rota-Baxter algebras of weight $\lambda$.
Proposition~\mref{prop:extd} gives a functor
\begin{equation}
 F_\omega: \OA_0 \to \ORB_\omega.
 \mlabel{eq:extfun}
\end{equation}

As a generalization of a differential Rota-Baxter algebra, we introduced in~\mcite{ZGK2} the concept of type $\omega$ differential Rota-Baxter algebras.

\begin{defn}
For a given $\omega\in\calt$ and $\lambda\in \bfk$, we say that the triple $(R, d, P)$ is a
{\bf type $\omega$ differential Rota-Baxter algebra of weight $\lambda$} if
\begin{enumerate}
\item $(R,d)$ is a differential algebra of weight $\lambda$,
\item $(R,P)$ is a Rota-Baxter algebra of weight $\lambda$, and
\item $\omega(d, P)=0$, that is,
   \begin{equation}
   dP=\phi(d)+P\psi(d).
   \mlabel{eq:dP1}
   \end{equation}
\end{enumerate}
\mlabel{de:DRBt}
\end{defn}
See~\cite[Example~3.5]{ZGK2} for examples of type $\omega$ differential Rota-Baxter algebras from analysis.

The category of type $\omega$ differential Rota-Baxter algebras of weight $\lambda$ will be denoted by $\DRB_\omega$. Note that $\DRB_\omega$ is a subcategory of $\ORB_\omega$ (resp., $\ODA_\omega$).

In the subsequent subsection, we will provide conditions ensuring that the restriction of the functor $G_\omega:\OA\to \ODA_\omega$ to the subcategory $\RBA$ of $\OA$ gives a functor $\RBA \to \DRB_\omega$. Likewise for $F_\omega$, as indicated in the following diagram.

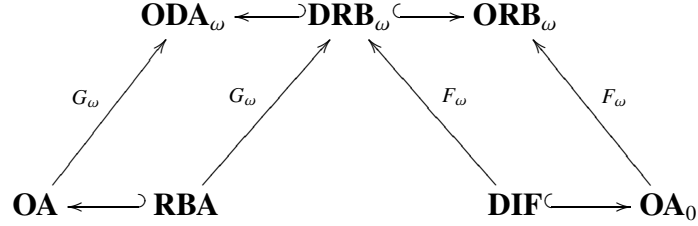
\begin{equation*}
\xymatrix{
&\ODA_\omega&\DRB_\omega\ar@{_{(}->}[l]\ar@{^{(}->}[r]&\ORB_\omega&\\
&&&&\\
\OA\ar^{G_\omega}[ruu]&\, \RBA \ar@{_{(}->}[l]\ar^{G_\omega}[ruu]&&\DIF \ar@{^{(}->}[r]\ar_{F_\omega}[luu] &\OA_0\ar_{F_\omega}[luu]
}
\mlabel{eq:cat}
\end{equation*}

\subsection{Main results}
\label{ss:main}

In this section, assume that $\bfk$ is a domain of characteristic $0$. As in ~(\mref{eq:T}), let
$ \calt:=xy+\bfk[x]+y\bfk[x].$
The following theorem shows that the existences of covers of Rota-Baxter operators, and of extensions of differential operators, of liftings of monads and comonads, and of mixed distributive laws are equivalent.

\begin{theorem}$($\cite[Theorem~3.15]{ZGK2}$)$
Let $\omega\in\calt$ be given. The following statements are equivalent:
\begin{enumerate}
\item
For every Rota-Baxter operator $P$ on every algebra $R$, the unique cover $\lpt$ of $P$ to $R^\NN$ given in Proposition~\mref{prop:extrb} is a Rota-Baxter operator.
\mlabel{it:extrb}
\item
For every differential operator $d$ on every algebra $R$, the unique extension $\ldt$ of $d$ to $\sha(R)$ given in Proposition~\mref{prop:extd} is a differential operator.
\mlabel{it:extd}
\item
The functor $G_\omega:\OA\to \ODA_\omega$ in ~(\mref{eq:coefun}) restricts to a functor $G_\omega:\RBA\to \DRB_\omega$.
\mlabel{it:coefun}
\item
The functor $F_\omega:\OA_0\to \ORB_\omega$ in ~(\mref{eq:extfun}) restricts to a functor $F_\omega:\DIF\to \DRB_\omega$.
\mlabel{it:extfun}
\item
There exist a comonad $\tilde{\C}= \langle \tilde{C},\tilde{\varepsilon},\tilde{\delta} \rangle$ on $\RBA$ lifting the comonad $\mathbf{C}$ from Proposition~\ref{prop:cofree}, where $\tilde{C}(R,P):=(R^{\NN}, \widetilde{P})$, and a category isomorphism $\widetilde{H}:\DRB_\omega\rightarrow\RBA_{\tilde{\C}}$ over $\RBA$ given by
$$\widetilde{H}(R, d, P):=\langle (R, P), \theta_{(R, d, P)}:(R, P)\rightarrow (R^{\NN}, \widetilde{P})\rangle.$$ Here $\theta_{(R, d, P)}(u)_{n}:=d^n(u)$ for all $u\in R, n\in\NN$.
\mlabel{it:liftcomo}
\item
There exist a monad $\tilde{\T}= \langle \tilde{T},\tilde{\eta},\tilde{\mu} \rangle$ on $\DIF$ lifting the monad $\mathbf{T}$ from Proposition~\ref{pp:adjrba}, where $\tilde{T}(R,d):=(\sha(R),\tilde{d})$, and a category isomorphism $\widetilde{K}:\DRB_\omega\rightarrow\DIF^{\tilde{\T}}$ over $\DIF$ given by $$\widetilde{K}(R, d, P):=\langle (R, d), \vartheta_{(R, d, P)}:(\sha(R),\tilde{d})\rightarrow (R, d)\rangle .$$
Here for every $v_0\ot v_1\ot\cdots \ot v_m\in \sha(R)$,
$$\vartheta_{(R, d, P)}(v_0\ot v_1\ot\cdots \ot v_m):=v_0P(v_1P(\cdots P(v_m)\cdots)).$$
\mlabel{it:liftmo}
\vspace{-0.15in}
\item
There is a mixed distributive law $\beta:TC\rightarrow CT$ such that $(\ALG_{\C})^{\tilde{\T}_\beta}$ is isomorphic to the category $\DRB_\omega$, where $\tilde{\T}_\beta$ is a lifting monad of $\T$ given by the mixed distributive law $\beta$.
\mlabel{it:mdl}
\end{enumerate}
\mlabel{thm:main}
\end{theorem}

It is important to classify in concrete terms the elements $\omega\in\calt$ that satisfy the equivalent conditions in the theorem. We obtain two results in this direction, first in the case when the weight is zero and then in the case when the weight is arbitrary. As we will see, the conditions imposed on $\omega$ are very strict.

Consider the following subsets of $\calt$:
$$\calt_0:=\{xy-a_0\,|\, a_0 \in \bfk\}\cup \{xy-\left(b_0y+yx\right)\,|\, b_0\in \bfk\}, \quad \calt_{\bfk}:=\{xy, xy-1, xy-yx\}.
$$
Then we give the main theorem in this paper:
\begin{theorem}
Let $\bfk$ be a domain of characteristic zero, and $\omega\in\calt$ be given.
\begin{enumerate}
\item
The equivalent conditions in Theorem~\ref{thm:main} hold in the case of $\lambda = 0$ if and only if $\omega$ is in $\Omega_0$.
\item
The equivalent conditions in Theorem~\ref{thm:main} hold for all weights $\lambda\in\bfk$ if and only if $\omega$ is in $\Omega_\bfk$.
\end{enumerate}
\mlabel{thm:lr0}
\end{theorem}

\section{Proof of the main theorem}
\mlabel{sec:proof}

Recall that all the seven statements in Theorem~\ref{thm:main} are equivalent for a given $\omega\in\calt$, and the first statement amounts to saying that the cover of a Rota-Baxter operator is again a Rota-Baxter operator. So we just need to prove the following theorem and then Theorem~\ref{thm:lr0} follows.

\begin{theorem}
Let $\omega=xy-(\phi(x)+y\psi(x))\in\calt$ be given.
\begin{enumerate}
\item
The following statements are equivalent.
\begin{enumerate}
\item
For every Rota-Baxter algebra $(R,P)$ of weight $0$, the cover $\lpt$ of $P$ on the differential algebra $(R^\NN,\partial_R)$ of weight $0$ given in Proposition~\mref{prop:extrb} is again a Rota-Baxter operator of weight $0$;
\item
$\omega$ is in $\calt_0.$
\end{enumerate}
\mlabel{it:weight0}
\item
The following statements are equivalent.
\begin{enumerate}
\item
For every Rota-Baxter algebra $(R,P)$ of arbitrary weight $\lambda$, the cover $\lpt$ of $P$ on the differential algebra $(R^\NN,\partial_R)$ of weight $\lambda$  given in Proposition~\mref{prop:extrb} is again a Rota-Baxter operator of weight $\lambda$;
\mlabel{it:b1}
\item
$\omega$ is in $\calt_{\bfk}.$
\mlabel{it:b2}
\end{enumerate}
\mlabel{it:weightlambda}
\end{enumerate}
\mlabel{thm:lr}
\end{theorem}

\begin{proof} (Summary)
The proof is divided into four parts. Let $\omega=xy-(\phi(x)+y\psi(x))$ with $\phi, \psi\in \bfk[x]$.
The first three parts cover the proof of Item~(\mref{it:weight0}), partitioned into the following three cases of $\omega$.

\begin{enumerate}
\item[]
{\bf Case 1. $\psi=0$.} This is proved in Section~\mref{ss:case1};
\item[]
{\bf Case 2. $\psi\neq 0$ and $\phi=0$.} This is proved in Section~\mref{ss:case2};
\item[]
{\bf Case 3. $\psi\neq 0$ and $\phi\neq 0$.} This is proved in Section~\mref{ss:case3}.
\end{enumerate}

The fourth part, given in Section~\mref{ss:part2}, proves Item~(\mref{it:weightlambda}) of the theorem.
\end{proof}

Based on this result and computations in some other cases, we propose the following

\begin{conjecture}
Let $\omega\in \bfk\langle x,y\rangle$ of the form
$\omega=xy-\sum\limits_{i=0}^\infty y^i\phi_i (x)$ be given.
\begin{enumerate}
\item
The following statements are equivalent.
\begin{enumerate}
\item
For every Rota-Baxter algebra $(R,P)$ of weight $0$, there is a cover $\lpt$ of $P$ on the differential algebra $(R^\NN,\partial_R)$ of weight $0$ that satisfies $\omega(\partial_R,\lpt)=0$ and is a Rota-Baxter operator of weight $0$;
\item
$\omega$ is in $\calt_0.$
\end{enumerate}
\mlabel{it:weight0}
\item
The following statements are equivalent.
\begin{enumerate}
\item
For every Rota-Baxter algebra $(R,P)$ of arbitrary weight $\lambda$, there is a cover $\lpt$ of $P$ on the differential algebra $(R^\NN,\partial_R)$ of weight $\lambda$ that satisfies $\omega(\partial_R,\lpt)=0$ and is a Rota-Baxter operator of weight $\lambda$;
\mlabel{it:b1}
\item
$\omega$ is in $\calt_{\bfk}.$
\mlabel{it:b2}
\end{enumerate}
\mlabel{it:weightlambda}
\end{enumerate}
\mlabel{conj:omega}
\end{conjecture}

Theorem~\mref{thm:lr} provides two classes of type $\omega$ differential Rota-Baxter algebras by the following corollary.
\begin{coro}
Fix an $\omega\in \calt_0$  $($resp.,  $\omega\in \calt_\bfk$$)$ and assume that $(R,P)$ is a Rota-Baxter algebra of weight $0$ (resp.,  weight $\lambda$) and $(R,d)$ is a differential algebra of weight $0$ (resp.,  weight $\lambda$), where $\lambda\in \bfk$ is arbitrary. Then we obtain two type $\omega$ differential Rota-Baxter algebra $(R^\NN,\partial_R, \lpt)$ and $(\sha(R), \ldt, P_R)$ of weight $0$ (resp.,  weight $\lambda$).
\end{coro}
\begin{proof}
Applying Theorem~\mref{thm:lr}, $\lpt$ is a Rota-Baxter operator of weight $0$ (resp.,  weight $\lambda$) on $R^\NN$. Also by ~(\mref{eq:lpproperty}), $\partial_R$ and $\lpt$ satisfy the required relation. Thus $(R^\NN,\partial_R, \lpt)$ is a type $\omega$ differential Rota-Baxter algebra of weight $0$ (resp.,  weight $\lambda$). Similarly, by Proposition~\ref{prop:extd} and Theorem~\ref{thm:main}, $(\sha(R), \ldt, P_R)$ is a type $\omega$ differential Rota-Baxter algebra.
\end{proof}
\smallskip

For a given $\omega=xy-(\phi(x)+y\psi(x))\in\calt$ with $\phi(x):=\sum\limits_{i=0}^{r}a_ix^i$, $\psi(x):=\sum\limits_{j=0}^{s}b_jx^j$, and each $f\in R^\NN, n\in \NN_+$, we obtain
\begin{eqnarray}
\lpt_{n}(f)&=&
(\phi(\partial_R)_{n-1}+\lpt_{n-1}\psi(\partial_R))(f)\quad\text{(by ~(\mref{eq:lptn}))}\notag\\
&=&\Big(\Big(\sum_{i=0}^ra_i \partial_R^i\Big)_{n-1}+ \sum_{j=0}^s b_j\lpt_{n-1}\partial_R^j\Big)(f)\notag\\
&=&\sum_{i=0}^ra_i f_{n-1+i} + \sum_{j=0}^s b_j\lpt_{n-1}(\partial_R^j f).
\mlabel{eq:lptn1}
\end{eqnarray}

Recall from Example~\mref{ex:couex} that, for $m\in \NN_+$,  $\overline{P_{\bfk}}$ is a Rota-Baxter operator on the quotient algebra $\sha(\bfk)/I_m$.
These Rota-Baxter algebras $(\sha(\bfk)/I_m, \overline{P_{\bfk}})$ will be used extensively with Corollary~\mref{co:qq0} to give counterexamples in the later proofs.

\subsection{Proof of Theorem~\mref{thm:lr}.(\mref{it:weight0}): Case 1}
\mlabel{ss:case1}

In this case $\omega:=xy-\phi(x)\in\calt$, where $\phi\in\bfk[x]$. Thus to prove Case 1 of Theorem~\mref{thm:lr}.(\mref{it:weight0}), we only need to prove the following proposition which provides an additional equivalent condition.
\begin{prop}
Let $\omega:=xy-\phi(x)$ with $\phi(x):=\sum\limits_{i=0}^{r}a_ix^i$.
The following statements are equivalent.
\begin{enumerate}
\item
For every Rota-Baxter algebra $(R,P)$ of weight $0$, the cover $\lpt$ of $P$ on the differential algebra $(R^\NN,\partial_R)$ of weight $0$ is again a Rota-Baxter operator of weight $0$;
\mlabel{it:case1a}
\item
$\phi=a_0$, that is, $\omega=xy-a_0$;
\mlabel{it:case1b}
\item
For every Rota-Baxter algebra $(R,P)$ of weight $0$, we have
\begin{equation}
\lpt(f)=(P(f_0), a_0 f_0, a_0 f_1, \cdots) \quad \text{for all }\ f\in R^\NN.
\mlabel{eq:lps}
\end{equation}
\mlabel{it:case1c}
\end{enumerate}
\label{pp:l}
\end{prop}
\begin{proof}
By ~(\mref{eq:lptn1}), we have
\begin{equation}
\lpt_{n}(f)=\sum\limits_{i=0}^ra_if_{n-1+i}\quad\text{for all}\  f\in R^\NN, n\in \NN_+.
\mlabel{eq:lpn}
\end{equation}
In particular, when $n=1$,
\begin{equation}
\lpt_{1}(f)=\sum\limits_{i=0}^ra_if_{i}\quad\text{for all}\  f\in R^\NN.
\mlabel{eq:lpn1}
\end{equation}

\noindent
(\mref{it:case1b}) $\Longrightarrow$ (\mref{it:case1c}).
When $\omega=xy-a_0$, by ~(\mref{eq:lpn}),
we obtain $\lpt_n(f)=a_0f_{n-1}$ for all $f\in R^\NN, n\in\NN_+$. Together with $\lpt_0(f)=P(f_0)$ from the definition of a cover, ~(\mref{eq:lps}) follows.
\smallskip

\noindent
(\mref{it:case1c}) $\Longrightarrow$ (\mref{it:case1b}).
Suppose $r:=\deg\,\phi\geq 1$, so $a_r\neq 0$. Take the Rota-Baxter algebra $(R, P)$ to be $(\sha(\bfk)/I_1, \overline{P_{\bfk}})$ in Example~\mref{ex:couex}, and let
$f:=(f_\ell) \in (\sha(\bfk)/I_1)^\NN$ with $f_\ell:=\delta_{\ell ,r}\overline{z_0}.$
Then ~(\mref{eq:lpn1}) gives $\lpt_1(f)=\sum\limits_{i=0}^{r}a_i\delta_{i, r}\overline{z_0}=a_r\overline{z_0}\neq \overline{0}$ while ~(\mref{eq:lps}) gives $\lpt_1(f)=a_0f_0=a_0\delta_{0, r}\overline{z_0}=\overline{0}$. This is a contradiction. Therefore, $\phi=a_0$.
\smallskip

\noindent
(\mref{it:case1a}) $\Longrightarrow$ (\mref{it:case1b}).
We just need to show that if $r:=\deg\,\phi\geq 1$, then there is a Rota-Baxter algebra $(R,P)$ such that the cover $\lpt$ of $P$ is not a Rota-Baxter operator on $R^\NN$.
When $r \geq 1$, we see $a_r\neq 0$. Let $M_n$ denote the maximum of the subscripts $m$ of the expressions $f_m$ appearing on the right hand side of ~(\mref{eq:lpn}). Then $M_n=n-1+r$.
Take $(R, P):=(\sha(\bfk)/I_2, \overline{P_{\bfk}})$ in Example~\mref{ex:couex}, and $f:=(f_\ell)\in (\sha(\bfk)/I_2)^\NN$ with $f_\ell:=\delta_{\ell, M_r}\overline{z_0}=\delta_{\ell, 2r-1}\overline{z_0}$.
For each $n\in \NN_+$ with $n\leq r$, ~(\mref{eq:lpn}) becomes
\begin{equation*}
\lpt_n(f)=\sum\limits_{i=0}^{r}a_i\delta_{n-1+i, 2r-1}\overline{z_0}= a_r\delta_{n-1+r, 2r-1}\overline{z_0}=
\begin{cases}
a_r\overline{z_0}, & \text{if}\ n=r, \\
\overline{0}, & \text{if} \ 1\leq n<r.
\end{cases}
\end{equation*}
Also by $\lpt_0(f)=P(f_0)=\overline{0}$, we have
\begin{equation}
\lpt_r(f)=a_r\overline{z_0},\quad\lpt_n(f)=\overline{0}\quad\text{for each}\ n\in\NN \ \text{with}\   n<r.
\mlabel{eq:Prr2}
\end{equation}
Let $g:=(g_k)\in (\sha(\bfk)/I_2)^\NN$ with $g_k:=\delta_{k, 0}\overline{z_0}$, i.e., $g$ is the identity element of $(\sha(\bfk)/I_2)^\NN$. Then ~(\mref{eq:lpn1}) and (\mref{eq:Prr2}) give
\begin{equation}
\lpt_1(\lpt(f)g)= \lpt_1(\lpt(f))=\sum_{i=0}^r a_i \lpt_i(f)
= a_r \lpt_r(f) = a_r^2 \overline{z_0}.
\mlabel{eq:fzl0}
\end{equation}
Since $r\leq 2r-1$, we have $f_i=\delta_{i,2r-1}\overline{z_0}=\overline{0}$ for each $i< r$. By ~(\mref{eq:lpn1}), $\lpt_1(f)=\sum\limits_{i=0}^{r}a_if_i=a_rf_r.$
Then we obtain
\begin{equation}
\lpt_0(f)\lpt_1(g)+\lpt_1(f)\lpt_0(g) =a_rf_r \lpt_0(g).
\mlabel{eq:fzl1}
\end{equation}
By ~(\mref{eq:lpn1}), $\lpt_1(f\lpt(g))=\sum\limits_{i=0}^{r}a_i(f\lpt(g))_i$. So applying ~(\mref{eq:hurprod0}), we have
$$\lpt_1(f\lpt(g)) =\sum\limits_{i=0}^{r}a_i\sum\limits_{j=0}^{i}{i\choose j} f_j \lpt_{i-j}(g).$$
Applying $f_j=\overline{0}$ for each $j< r$ again, we obtain
\begin{equation}
 \lpt_1(f\lpt(g)) =a_rf_r \lpt_0(g).
\mlabel{eq:fzl2}
\end{equation}
Combining ~(\mref{eq:fzl0}), (\mref{eq:fzl1}) and (\mref{eq:fzl2}), we obtain
\begin{eqnarray*}
\left(\lpt_1(\lpt(f)g)+\lpt_1(f\lpt(g))\right)- \left(\lpt_0(f)\lpt_1(g)+\lpt_1(f)\lpt_0(g)\right)
&=&(a_r^2\overline{z_0}+a_rf_r \lpt_0(g))-a_rf_r \lpt_0(g)\\
&=&a_r^2\overline{z_0}\neq \overline{0}.
\end{eqnarray*}
Thus by Corollary~\mref{co:qq0}, $\lpt$ is not a Rota-Baxter operator on $R^\NN$.

\smallskip

\noindent
(\mref{it:case1b}) $\Longrightarrow$ (\mref{it:case1a}).
By Proposition~\mref{prop:qq}.(\mref{it:pqn}), we need to show that for any Rota-Baxter algebra $(R, P)$, and all $f, g\in R^\NN$, $n\in\NN$,
$$\sum_{k=0}^{n} {n\choose k}
\lpt_k(f)\lpt_{n-k}(g)=\lpt_n(\lpt(f)g)+
\lpt_n(f\lpt(g))$$
holds. Applying Proposition~\mref{prop:qq}.(\mref{it:pq0}), we have
$$\lpt_0(f)\lpt_0(g)=\lpt_0(\lpt(f)g)+
\lpt_0(f\lpt(g)).$$
Since $\omega=xy-a_0$, ~(\mref{eq:lpn}) gives
\begin{equation}
\lpt_n(h)=a_0h_{n-1}\quad \text{for all} \  h\in R^\NN, n\in\NN_+.
\mlabel{eq:a0Pn}
\end{equation}
Then
\begin{eqnarray}
\lpt_n(\lpt(f)g)
&=&a_0(\lpt(f)g)_{n-1}\quad \text{(by ~(\mref{eq:a0Pn}))}\notag\\
&=&\lpt_0(f)(a_0 g_{n-1})
+\sum_{k=1}^{n-1} {n-1\choose k}
\lpt_k(f)(a_0g_{n-k-1})\quad \text{(by ~(\mref{eq:hurprod0}))}\notag\\
&=&\lpt_0(f)\lpt_n(g)
+\sum_{k=1}^{n-1}  {n-1\choose k}
\lpt_k(f)\lpt_{n-k}(g).\quad \text{(by ~(\mref{eq:a0Pn}))}\mlabel{eq:fzl3}
\end{eqnarray}
Exchanging $f$ and $g$, and then applying the commutativity of the multiplication, we obtain
\begin{eqnarray}
\lpt_n(f\lpt(g))
&=&\lpt_n(f)\lpt_0(g)
+\sum_{k=1}^{n-1} {n-1\choose k}
\lpt_{n-k}(f)\lpt_{k}(g)\notag\\
&=&\lpt_n(f)\lpt_0(g)
+\sum_{k=1}^{n-1} {n-1\choose n-k}
\lpt_{k}(f)\lpt_{ n-k}(g)
\mlabel{eq:fzl4}
\end{eqnarray}
by  exchanging $k$ and $n-k$.
Combining ~(\mref{eq:fzl3}) and (\mref{eq:fzl4}), and ${n\choose k}= {n-1\choose k}+ {n-1\choose n-k}$, we obtain
\begin{eqnarray*}
\sum_{k=0}^{n} {n\choose k}
\lpt_k(f)\lpt_{n-k}(g)&=&\lpt_0(f)\lpt_n(g)+\lpt_n(f)\lpt_0(g)+\sum_{k=1}^{n-1} {n\choose k}
\lpt_k(f)\lpt_{n-k}(g)\\
&=&\lpt_n(\lpt(f)g)+\lpt_n(f\lpt(g)),
\end{eqnarray*}
as required.
\end{proof}

\subsection{Proof of Theorem~\mref{thm:lr}.(\mref{it:weight0}): Case 2}
\mlabel{ss:case2}
In this case $\omega:=xy-y\psi(x)\in\calt$ with  $\psi\in\bfk[x]$. Thus to prove the Case 2 of Theorem~\mref{thm:lr}.(\mref{it:weight0}), we only need to prove the following strengthened form.
\begin{prop}
Let $\omega:=xy-y\psi(x)$ with $\psi(x):=\sum\limits_{j=0}^{s}b_jx^j\neq 0$.
The following statements are equivalent.
\begin{enumerate}
\item
For every Rota-Baxter algebra $(R,P)$ of weight $0$, the cover $\lpt$ of $P$ on the differential algebra $(R^\NN,\partial_R)$ of weight $0$ is again a Rota-Baxter operator of weight $0$;
\mlabel{it:case2a}
\item
$\deg\,\psi=1$ and $b_1=1$, that is, $\omega=xy-(b_0y+yx)$;
\mlabel{it:case2b}
\item
For every Rota-Baxter algebra $(R,P)$ of weight $0$ and each $f\in R^\NN$,
we have
$$\lpt(f)=(\lpt_0(f), \lpt_1(f), \cdots, \lpt_{n}(f),\cdots),$$ where $\lpt_0(f)=P(f_0)$ and for each $n\in\NN_+$, $\lpt_{n}(f)$ is given recursively by
\begin{equation}
\lpt_{n}(f)=b_0\lpt_{n-1}(f)+\lpt_{n-1}(\partial_R f).
\mlabel{eq:rps}
\end{equation}
In particular, if $b_0=0$, then
$$\lpt_{n}(f)=\lpt_{n-1}(\partial_R f)=\cdots=\lpt_{0}(\partial_R^n f)=P(f_n).$$
That is,
$$\lpt(f)=(P(f_0), P(f_1), P(f_2), \cdots).$$
\label{it:case2c}
\end{enumerate}
\label{pp:r}
\end{prop}

\begin{proof}
Recall from  ~(\mref{eq:lptn1}) that the cover $\lpt$ is given by
\begin{equation}
\lpt_{n}(f)=\sum\limits_{j=0}^sb_j\lpt_{n-1}(\partial_R^j f) \quad \text{for all}\  f\in R^\NN, n\in \NN_+.
\mlabel{eq:recPn(f)}
\end{equation}
In particular,
\begin{equation}
\lpt_1(f)=\sum_{j=0}^sb_j\lpt_0(\partial_R^j f)= \sum_{j=0}^sb_jP(f_j)\quad\text{for all}\  f\in R^\NN.
\mlabel{eq:rP1(f)}
\end{equation}
In general, by iterating ~(\mref{eq:recPn(f)}), we obtain
\begin{equation}
\lpt_{n}(f)= \sum\limits_{j_{1}=0}^sb_{j_{1}}\sum\limits_{j_{2}=0}^sb_{j_{2}}\cdots
\sum\limits_{j_{n}=0}^sb_{j_{n}}\lpt_0({\partial_R}^{j_{1}+j_{2}+\cdots+j_{n}}f)
=\sum\limits_{j_{1}, j_{2},\cdots,j_{n}=0}^sb_{j_{1}} b_{j_{2}}\cdots b_{j_{n}}P(f_{j_{1}+j_{2}+\cdots+j_n}).
\mlabel{eq:rPn(f)}
\end{equation}

\noindent
(\mref{it:case2b}) $\Longrightarrow$ (\mref{it:case2c}). For any $f\in R^\NN$, $\lpt_0(f)=P(f_0)$ follows from the definition of a cover. By $\omega:=xy-(b_0y+yx)$ and ~(\mref{eq:recPn(f)}), we obtain $\lpt_n(f)=b_0\lpt_{n-1}(f)+\lpt_{n-1}(\partial_R f)$ for all $n\in\NN_+.$
\smallskip

\noindent
(\mref{it:case2c}) $\Longrightarrow$ (\mref{it:case2b}).
Assume that Item~(\mref{it:case2c}) holds. Suppose $s:=\deg\,\psi\geq 2$. Take $(R, P):=(\sha(\bfk)/I_2, \overline{P_{\bfk}})$ in Example~\mref{ex:couex}. Let $f:=(f_k) \in (\sha(\bfk)/I_2)^\NN$ with $f_k:=\delta_{k ,s}\overline{z_0}.$
Then ~(\mref{eq:rps}) gives
$$\lpt_1(f)=b_0\lpt_0(f)+\lpt_0(\partial_R f)=b_0P(f_0)+P(f_1)=b_0P(\delta_{0 ,s}\overline{z_0})+P(\delta_{1 ,s}\overline{z_0})=\overline{0}$$
while $\lpt_1(f)=\sum\limits_{j=0}^{s}b_jP(\delta_{j ,s}\overline{z_0})=b_s\overline{z_1}\neq \overline{0}$ by ~(\mref{eq:rP1(f)}). This is a contradiction. Thus $s=\deg\,\psi\leq 1$.

Now take $f:=(f_\ell) \in (\sha(\bfk)/I_2)^\NN$ with $f_\ell:=\delta_{\ell ,1}\overline{z_0}.$
Then ~(\mref{eq:rps}) gives
$$\lpt_1(f)=b_0\lpt_0(f)+\lpt_0(\partial_R f)=b_0P(f_0)+P(f_1)=b_0P(\delta_{0 ,1}\overline{z_0})+P(\delta_{1 ,1}\overline{z_0})=\overline{z_1}$$
while ~(\mref{eq:rP1(f)}) gives
\begin{equation*}
\lpt_1(f)=\sum\limits_{j=0}^{s}b_jP(\delta_{j ,1}\overline{z_0})=
\begin{cases}
b_1P(\overline{z_0})=b_1\overline{z_1}, & \text{if}\ s=1, \\
\overline{0}, & \text{if} \ s=0.
\end{cases}
\end{equation*}
Thus we obtain $s=1$ and $\overline{z_1}=b_1\overline{z_1}$. Then $b_1=1$ since $\overline{z_1}$ is one of the basis elements. Therefore, $\omega=xy-(b_0y+yx)$.
\smallskip

\noindent
(\ref{it:case2a}) $\Longrightarrow$ (\ref{it:case2b}).
Let $s:=\deg\,\psi$.
Consider $(R, P):=(\sha(\bfk)/I_3, \overline{P_{\bfk}})$ and take $g:=(g_k)\in (\sha(\bfk)/I_3)^\NN$ with $g_k:=\delta_{k, 0}\overline{z_0}$, i.e., $g$ is the identity element of $(\sha(\bfk)/I_3)^\NN$. Then $\lpt_1(\lpt(f)g)=\lpt_1(\lpt(f))$.
So applying ~(\mref{eq:rP1(f)}), we obtain
\begin{equation}
\lpt_1(\lpt(f)g)=\sum_{j=0}^s b_j P(\lpt_j(f)).
\mlabel{eq:1lra}
\end{equation}

Suppose $s=0$, i.e., $\psi=b_0$. Then ~(\mref{eq:1lra}) becomes $\lpt_1(\lpt(f)g)=b_0 P(\lpt_0(f))$. Now let $f:=(f_\ell)\in (\sha(\bfk)/I_3)^\NN$ with $f_\ell:=\delta_{\ell, 0}\overline{z_0}$. Applying $f=g$ and the commutativity of the multiplication, we have
\begin{equation}
\lpt_1(f\lpt(g))=\lpt_1(\lpt(f)g)=b_0 P(P(f_0))=b_0\overline{z_2}.
\mlabel{eq:fz1a}
\end{equation}
Applying ~(\mref{eq:tmn}) and (\mref{eq:rP1(f)}), we obtain
\begin{equation}
\lpt_0(f)\lpt_1(g)+\lpt_1(f)\lpt_0(g)=P(f_0) b_0P(g_0)+b_0P(f_0)P(g_0)=2b_0\overline{z_1}^2=4b_0\overline{z_2}.
\mlabel{eq:fz2a}
\end{equation}
Combining ~(\mref{eq:fz1a}) and (\mref{eq:fz2a}) gives
$$\left(\lpt_0(f)\lpt_1(g)+\lpt_1(f)\lpt_0(g)\right)-
\left(\lpt_1(\lpt(f)g)+\lpt_1(f\lpt(g))\right)=
4b_0\overline{z_2}-2b_0\overline{z_2}=2b_0\overline{z_2}\neq \overline{0}.$$
So $\lpt$ is not a Rota-Baxter operator on $R^\NN$ by Corollary~\mref{co:qq0}. So we must have $s\geq 1$.

Now let $s\geq 1$ be given. Then $b_s\neq 0$. Let $M_n$ denote the maximum of the subscripts $m$ of the expressions $f_m$ appearing on the right hand side of ~(\mref{eq:rPn(f)}):
$$M_n:=\max\{j_1+j_2+\cdots+j_n \, |\, 0\leq j_1, j_2, \cdots, j_n\leq s\}= ns.$$
Take $f:=(f_\ell) \in (\sha(\bfk)/I_3)^\NN$ with $f_\ell:=\delta_{\ell ,M_s}\overline{z_0}=\delta_{\ell ,s^2}\overline{z_0}$.
For each $n\in\NN_+$ with $n\leq s$, ~(\mref{eq:rPn(f)}) becomes
\begin{equation*}
\lpt_n(f)=\sum\limits_{j_{1}, j_{2},\cdots,j_{n}=0}^sb_{j_{1}} b_{j_{2}}\cdots b_{j_{n}}P(\delta_{j_{1}+j_{2}+\cdots+j_n, s^2}\overline{z_0})=b_s^nP(\delta_{ns, s^2}\overline{z_0})=
\begin{cases}
b_s^sP(\overline{z_0})=b_s^s\overline{z_1}, & \text{if}\ n=s, \\
\overline{0}, & \text{if} \ 1\leq n<s.
\end{cases}
\end{equation*}
Together with $\lpt_0(f)=P(f_0)=P(\delta_{0,s^2}\overline{z_0})=\overline{0}$ from $s^2>0$, we obtain
\begin{equation}
\lpt_s(f)=b_s^s\overline{z_1},\quad \lpt_n(f)=\overline{0}\quad\text{for each}\ n\in\NN \ \text{with}\  n< s.
\mlabel{eq:0lrs}
\end{equation}
Then ~(\mref{eq:1lra}) gives
\begin{equation}
\lpt_1(\lpt(f)g)=b_s P(\lpt_s(f))= b_s P(b_s^s \overline{z_1}) =b_s^{s+1} \overline{z_2}.
\mlabel{eq:1lr}
\end{equation}
Also by ~(\mref{eq:rP1(f)}) and (\ref{eq:hurprod0}), we have
\begin{equation}
\lpt_1(f\lpt(g))=\sum\limits_{j=0}^sb_jP((f\lpt(g))_j)=
\sum\limits_{j=0}^sb_jP\left(\sum\limits_{i=0}^{j}{j\choose i}f_i\lpt_{j-i}(g)\right) =\left\{\begin{array}{ll} \overline{0}, & \text{if}\ s\geq 2, \\ b_1\overline{z_2}, & \text{if}\ s=1.\end{array} \right.
\mlabel{eq:r20a}
\end{equation}
Here the last equation follows from $f_i=\delta_{i,s^2}\overline{z_0}$ since $0\leq i\leq j\leq s\leq s^2$ with equality holding in the last inequality if and only if $s=1$.
Further applying ~(\mref{eq:0lrs}), we have
\begin{equation}
\lpt_0(f)\lpt_1(g)+\lpt_1(f)\lpt_0(g)=\left\{\begin{array}{ll} \overline{0}, & \text{if}\ s\geq 2, \\ b_1\overline{z_1}\lpt_0(g)=b_1\overline{z_1}^2 = 2b_1\overline{z_2}, & \text{if}\ s=1. \end{array}\right.
\mlabel{eq:rzz}
\end{equation}
Combining ~(\ref{eq:1lr}), (\ref{eq:r20a}) and (\ref{eq:rzz}), we obtain
\begin{equation*}
\left(\lpt_1(\lpt(f)g)+\lpt_1(f \lpt(g))\right)-\left(\lpt_0(f)\lpt_1(g)+\lpt_1(f)\lpt_0(g)\right)= \left\{\begin{array}{ll} b_s^{s+1}\overline{z_2}\neq \overline{0}, &\text{if}\ s\geq 2, \\ b_1(b_1-1)\overline{z_2}&\text{if}\  s=1. \end{array} \right.
\mlabel{eq:f01}
\end{equation*}
Then by Corollary~\mref{co:qq0}, when $s\geq 2$, $\lpt$ is not a Rota-Baxter operator.
When $s= 1$, we obtain $b_1-1=0$, i.e., $b_1=1$.

Therefore, we must have $s=1$ and $b_1=1$.
\smallskip

\noindent
(\ref{it:case2b}) $\Longrightarrow$ (\ref{it:case2a}).
Let $(R,P)$ be an arbitrary Rota-Baxter algebra $(R,P)$. We will prove that $\lpt$ is a Rota-Baxter operator on $R^\NN$ by verifying the componentwise formulation
\begin{equation}
(\lpt(f)\lpt(g))_{n}=\lpt_{n}(\lpt(f)g)+\lpt_{n}(f\lpt(g)) \quad \text{for all } f, g\in R^\NN, n\in\NN,
\mlabel{eq:induction}
\end{equation}
of the Rota-Baxter relation in ~(\mref{eq:bax1}). We will carry out the verification by induction on $n$.

First by Proposition~\mref{prop:qq}.(\mref{it:pq0}), we have
$$(\lpt(f)\lpt(g))_{0}=\lpt_{0}(\lpt(f)g)+\lpt_{0}(f\lpt(g)).$$

Assume that for a given $k\in\NN$, ~(\mref{eq:induction}) holds.
Then we derive
\begin{eqnarray*}
&&\lpt_{k+1}(\lpt(f)g)+\lpt_{k+1}(f\lpt(g))\\
&=&(b_0\lpt_k+\lpt_k\partial_R)(\lpt(f)g+f\lpt(g))\quad \text{(by ~(\mref{eq:recPn(f)}))}\\
&=&b_0\lpt_k\left(\lpt(f)g+f\lpt(g)\right)+\lpt_k\left((\partial_R\lpt)(f)g
+\lpt(f)\partial_R(g)\right)\\
&&+\lpt_k\left(\partial_R(f)\lpt(g)+f(\partial_R\lpt)(g)
\right)\ \text{(by ~(\mref{eq:der10}))}\\
&=&b_0\lpt_k\left(\lpt(f)g+f\lpt(g)\right)+\lpt_k\left((b_0\lpt+\lpt\partial_R)(f)g
+\lpt(f)\partial_R(g)\right)\\
&&+\lpt_k\left(\partial_R(f)\lpt(g)+f(b_0\lpt+\lpt\partial_R)(g)
\right)\quad\text{(by ~(\mref{eq:lpproperty}))}\\
&=&2b_0\left(\lpt(f)\lpt(g)\right)_k+\left(\lpt(\partial_Rf)\lpt(g)\right)_k
+\left(\lpt(f)\lpt(\partial_Rg)\right)_k\quad\text{(by the induction hypothesis)}\\
&=&\left((b_0\lpt+\lpt\partial_R)(f)\lpt(g)\right)_k+
\left(\lpt(f)(b_0\lpt+\lpt\partial_R)(g)\right)_k\\
&=&\left((\partial_R\lpt)(f)\lpt(g)+
\lpt(f)(\partial_R\lpt)(g)\right)_k\quad\text{(by ~(\mref{eq:lpproperty}))}\\
&=&\left(\lpt(f)\lpt(g)\right)_{k+1}\quad\text{(by ~(\mref{eq:recHur}))}.
\end{eqnarray*}
This completes the induction.
\end{proof}

\subsection{Proof of Theorem~\mref{thm:lr}.(\mref{it:weight0}): Case 3}
\mlabel{ss:case3}

In this case, $\omega:=xy-(\phi(x)+y\psi(x))\in\calt$, where $\phi, \psi\in\bfk[x]$ are nonzero with $r:=\deg\,\phi, s:=\deg\,\psi\in \NN$.
To prove Theorem~\mref{thm:lr}.(\mref{it:weight0}) in this case, we will apply the same idea as in the previous two cases, namely by taking the maximum of the subscripts. But in order for the idea to work, we need to partition $\NN^2$ into eight subsets before carrying out the proof in Proposition~\mref{pp:lr}.

Let $(R, P)$ denote an arbitrary Rota-Baxter algebra, and $\phi(x):=\sum\limits_{i=0}^{r}a_ix^i$ and $\psi(x):=\sum\limits_{j=0}^{s}b_jx^j$. As in ~(\mref{eq:lptn1}), we have
\begin{equation}
\lpt_{n}(f)=\sum_{i=0}^ra_i f_{n-1+i} + \sum_{j=0}^s b_j\lpt_{n-1}(\partial_R^j f)\quad \text{for all}\  f\in R^\NN, n\in \NN_+.
\mlabel{eq:recur3}
\end{equation}
In particular, if $n=1$, then ~(\mref{eq:recur3}) becomes
\begin{equation}
\lpt_1(f)=\sum\limits_{i=0}^ra_i f_{i} + \sum\limits_{j=0}^s b_jP(f_j).
\mlabel{eq:3P1}
\end{equation}
Expanding the recursion in ~(\mref{eq:recur3}), we obtain
\begin{eqnarray}
\lpt_{n}(f)&=&\sum^r_{i=0}a_{i} f_{n-1+i} + \sum^s_{j_1=0} b_{j_1}\lpt_{n-1}(\partial_R^{j_1} f)\quad\text{(by ~(\mref{eq:recur3}))}\notag\\
&=&\sum^r_{i=0}a_{i} f_{n-1+i} + \sum^s_{j_1=0}b_{j_1}\left(
\sum^r_{i=0}a_{i} f_{n-2+i+j_1} + \sum^s_{j_2=0} b_{j_2}\lpt_{n-2}(\partial_R^{j_1+j_2} f)
\right)\quad\text{(by ~(\mref{eq:recur3}))}\notag\\
&=& \sum^r_{i=0}a_{i} f_{n-1+i} + \sum^r_{i=0}\sum^s_{j_1=0} a_{i}b_{j_1}  f_{n-2+i+j_1} + \sum^s_{j_1,j_2=0} b_{j_1}b_{j_2}\lpt_{n-2}(\partial_R^{j_1+j_2} f).\notag
\end{eqnarray}
Repeating this process leads to
\begin{equation}
\lpt_n(f)=\sum^r_{i=0}\sum_{k=0}^{n-1}\sum^s_{j_1,\cdots,j_{k}=0}a_ib_{j_1}\cdots b_{j_{k}} f_{n-1-k+i+j_1+\cdots+j_{k}} +\sum^s_{j_1, \cdots,j_{n}=0}b_{j_1}\cdots b_{j_{n}}P(f_{j_1+ \cdots+j_{n}})
\mlabel{eq:3Pn}
\end{equation}
for all $f\in R^\NN, n\in\NN_+$.

Let $M_n$ denote the maximum of the subscripts of the expressions $f_m$ appearing on the right hand side of ~(\mref{eq:3Pn}):
$$
M_n:
=\max\left\{n-1-k+i+j_1+\cdots +j_k,\, j_1+\cdots +j_n\,|\, 0\leq k\leq n-1, \,0\leq i\leq r,\, 0\leq j_1, \cdots, j_n\leq s\right\}. $$
By first partitioning $s\in\NN$ into $s>1$ , $s=1$ and $s<1$ (that is $s=0$) and then partitioning each of the three cases into the subcases of $r>s$, $r=s$ and $r<s$ (the latter subcase is valid only when $s>1$ and $s=1$), we partition $(r,s)\in \NN^2$ into eight cases in the following lemma.

\begin{lemma}
Let $n\in\NN_+$ and $f:=(f_\ell)\in R^\NN$ with $f_\ell:=\delta_{\ell,M_n}u$, where $u$ is a given nonzero element in $R$. The possibilities of $M_n$ and $\lpt_\sigma(f)$ for all $\sigma \leq n$ are as follows.
\begin{enumerate}
\item
If $s>1$ and $r>s$, then $M_n=r+(n-1)s$, $\lpt_n(f)=a_rb_s^{n-1}u$ and $\lpt_\sigma(f)=0$ for  $\sigma< n$;
\mlabel{it:t1}
\item
If $s>1$ and $r=s$, then $M_n =ns$, $\lpt_n(f)=a_rb_s^{n-1}u+b_s^nP(u)$ and $\lpt_\sigma(f)=0$ for  $\sigma< n$;
\mlabel{it:t2}
\item
If $s>1$ and $r<s$, then $M_n =ns$, $\lpt_n(f)=b_s^nP(u)$ and $\lpt_\sigma(f)=0$ for  $\sigma< n$;
\mlabel{it:t3}
\item
If $s=1$ and $r>s$, then $M_n=n-1+r$, $\lpt_n(f)=\sum\limits_{k=0}^{n-1} a_rb_{1}^k u$ and $\lpt_\sigma(f)=0$ for  $\sigma< n$;
\mlabel{it:t4}
\item
If $s=1$ and $r=s$, then $M_n=n$, $\lpt_n(f)=\sum\limits_{k=0}^{n-1} a_rb_{1}^k u +b_{1}^n P(u)$ and $\lpt_\sigma(f)=0$ for  $\sigma< n$;
\mlabel{it:t5}
\item
If $s=1$ and $r<s$, then $M_n=n$, $\lpt_n(f)=b_{1}^n P(u)$ and $\lpt_\sigma(f)=0$ for  $\sigma< n$;
\mlabel{it:t6}
\item
If $s=0$ and $r>s$, then $M_n=n-1+r$, $\lpt_n(f)=a_r u$ and $\lpt_\sigma(f)=0$ for  $\sigma< n$;
\mlabel{it:t7}
\item
If $s=0$ and $r=s$, then $M_n=n-1$, $\lpt_n(f)=a_r u+\delta_{n, 1}b_{0} P(u)$,  $\lpt_\sigma(f)=\delta_{n, 1}P(u)$ for  $\sigma< n$. \mlabel{it:t8}
\end{enumerate}
\mlabel{lem:cases}
\end{lemma}
\begin{proof}
By the choice of $f$, ~(\mref{eq:3Pn}) becomes
$$
\lpt_n(f)=\sum^r_{i=0}\sum_{k=0}^{n-1}\sum^s_{j_1, \cdots,j_{k}=0}a_ib_{j_1} \cdots b_{j_{k}} \delta_{n-1-k+i+j_1+ \cdots+j_{k}, M_n}u
+\sum^s_{j_1, \cdots,j_{n}=0}b_{j_1} \cdots b_{j_{n}}P(\delta_{j_1+  \cdots+j_{n}, M_n}u).
$$
Since the two indices of the Kronecker deltas are possibly equal only when $i$ and $j_1, \cdots, j_n$ are maximized, we have
\begin{eqnarray}
\lpt_n(f)&=&\sum_{k=0}^{n-1} a_rb_{s}^k \delta_{n-1-k+r+ks, M_n}u +b_{s}^nP(\delta_{ns, M_n}u)\notag\\
&=&\sum_{k=0}^{n-1} a_rb_{s}^k \delta_{n-1+r+k(s-1), M_n}u +b_{s}^n\delta_{n-1+s+(n-1)(s-1), M_n}P(u).
\mlabel{eq:M00}
\end{eqnarray}

We first prove the first and second equations in all the cases of the lemma.

When $s>1$, namely $s-1>0$, by maximizing $k$, ~(\mref{eq:M00}) becomes
\begin{eqnarray}
\lpt_n(f)&=&a_rb_{s}^{n-1} \delta_{n-1+r+(n-1)(s-1), M_n}u +b_{s}^n\delta_{n-1+s+(n-1)(s-1), M_n}P(u)\notag\\
&=&a_rb_{s}^{n-1} \delta_{r+(n-1)s, M_n}u +b_{s}^n\delta_{s+(n-1)s, M_n}P(u).\notag
\end{eqnarray}
Thus when $r>s$  (resp.,  $r=s$, resp., $r<s$), we obtain $M_n=r+(n-1)s$ (resp.,  $M_n=ns$, resp., $M_n=ns$) and $\lpt_n(f)=a_rb_s^{n-1}u$  (resp.,  $\lpt_n(f)=a_rb_s^{n-1}u+b_s^nP(u)$, resp., $\lpt_n(f)=b_s^nP(u)$), proving the first and second equations in cases~(\mref{it:t1}) -- (\mref{it:t3}) of the lemma.

When $s=1$, namely $s-1=0$, (\mref{eq:M00}) becomes
$$\lpt_n(f)=\sum_{k=0}^{n-1} a_rb_{1}^k \delta_{n-1+r, M_n}u +b_{1}^n\delta_{n-1+s, M_n}P(u).$$
Thus when $r>s=1$  (resp.,  $r=s=1$, resp., $r<s=1$), we obtain $M_n=n-1+r$ (resp.,  $M_n=n$, resp., $M_n=n$) and $\lpt_n(f)=\sum\limits_{k=0}^{n-1} a_rb_{1}^k u$  (resp.,  $\lpt_n(f)=\sum\limits_{k=0}^{n-1} a_rb_{1}^k u +b_{1}^n P(u)$, resp., $\lpt_n(f)=b_{1}^n P(u)$), proving the first and second equations in Item~(\mref{it:t4}) -- (\mref{it:t6}) of the lemma.

When $s<1$, namely $s=0$ and $s-1=-1$, by minimizing $k$, ~(\mref{eq:M00}) becomes
$$\lpt_n(f)=a_r\delta_{n-1+r, M_n}u +b_{0}^n\delta_{0, M_n}P(u).$$
Thus when $r>s=0$  (resp.,  $r=s=0$), we obtain $M_n=n-1+r$  (resp.,  $M_n=n-1$) and $\lpt_n(f)=a_r u$  (resp.,  $\lpt_n(f)=a_r u+\delta_{n, 1}b_{0} P(u)$), proving the first and second equations in Item~(\mref{it:t7}) -- (\mref{it:t8}) of the lemma.

Now we prove the third equations in all the cases of the lemma.
In each of the cases~(\mref{it:t1}) -- (\mref{it:t7}), since $M_n>0$, we have $\lpt_0(f)=P(f_0)=P(\delta_{0, M_n})=0$. For case~(\mref{it:t8}), $$\lpt_0(f)=P(f_0)=P(\delta_{0, M_n}u)=P(\delta_{0, n-1}u)=\delta_{n, 1}P(u).$$ This proves the third equations when $\sigma=0$.

In each of the cases~(\mref{it:t1}) -- (\mref{it:t8}), take $\sigma$ with $1\leq \sigma <n$. Then $n>1$ and so $M_\sigma <M_n$. Thus the expressions $f_\tau$ appearing in $\lpt_\sigma(f)$ all vanish since the subscripts of the expressions are strictly smaller than $M_n$. Therefore, $\lpt_\sigma(f)=0$.
This completes the proof of Lemma~\mref{lem:cases}.
\end{proof}

We also need the following facts to proceed.

For a Rota-Baxter algebra $(R, P)$, take $f:=(f_\ell)$ and $g:=(g_k)$ in $R^\NN$.
Then
\begin{eqnarray}
\lpt_1(\lpt(f)g)&=&\sum\limits_{i=0}^ra_i (\lpt(f)g)_{i} + \sum\limits_{j=0}^s b_jP((\lpt(f)g)_j)\quad\text{(by ~(\mref{eq:3P1}))}\notag\\
&=&\sum\limits_{i=0}^ra_i \sum_{\sigma=0}^{i} {i\choose \sigma}\lpt_\sigma(f)g_{i-\sigma} + \sum\limits_{j=0}^s b_jP\left(\sum_{\tau=0}^{j} {j\choose \tau}\lpt_\tau(f)g_{j-\tau}\right)\quad\text{(by ~(\mref{eq:hurprod0}))}
\mlabel{eq:RBr10}
\end{eqnarray}
and
\begin{eqnarray}
\lpt_1(f\lpt(g))&=&\sum\limits_{i=0}^ra_i (f\lpt(g))_{i} + \sum\limits_{j=0}^s b_jP((f\lpt(g))_j)\quad\text{(by ~(\mref{eq:3P1}))}\notag\\
&=&\sum\limits_{i=0}^ra_i \sum_{\sigma=0}^{i} {i\choose \sigma}f_{\sigma}\lpt_{i-\sigma}(g) + \sum\limits_{j=0}^s b_jP\left(\sum_{\tau=0}^{j} {j\choose \tau}f_{\tau}\lpt_{j-\tau}(g)\right).\quad\text{(by ~(\mref{eq:hurprod0}))}
\mlabel{eq:RBr2}
\end{eqnarray}
Let $N_f$ denote the maximal subscript of expressions $f_m$ appearing in the right hand side of ~(\mref{eq:RBr2}):
\begin{equation}
N_f:
=\max\left\{\sigma,\, \tau\,|\, \sigma\leq i\leq r, \,\tau\leq j\leq s\right\}=\max\{r,\, s\}.
\mlabel{eq:Nf}
\end{equation}

Now take $(R,P):=(\sha(\bfk)/I_m, \overline{P_{\bfk}})$ from Example~\mref{ex:couex}.
Let $g=(g_k)$ with $g_k:=\delta_{k, 0}\overline{z_0}$, i.e., $g$ is the identity element of $(\sha(\bfk)/I_m)^\NN$. Then by ~(\mref{eq:3P1}),
\begin{equation}
\lpt_1(\lpt(f)g)=\lpt_1(\lpt(f))=\sum\limits_{i=0}^ra_i \lpt_i(f) + \sum\limits_{j=0}^s b_jP(\lpt_j(f)).
\mlabel{eq:RBr11}
\end{equation}
~(\mref{eq:3P1}) also gives
\begin{equation*}
\lpt_1(g)=a_0g_0+b_0P(g_0)=a_0\overline{z_0}+b_0\overline{z_1}.
\end{equation*}
Then
\begin{equation}
\lpt_0(f)\lpt_1(g)+\lpt_1(f)\lpt_0(g) =
\lpt_0(f)(a_0\overline{z_0}+b_0\overline{z_1})+\lpt_1(f)\overline{z_1}.
\mlabel{eq:RBl0}
\end{equation}

Now we are ready to prove Case~3 of Theorem~\mref{thm:lr}.(\mref{it:weight0}).
\begin{prop}
For each $\omega:=xy-(\phi(x)+y\psi(x))\in \calt$ with nonzero $\phi, \psi\in \bfk[x]$,
there is a Rota-Baxter algebra $(R, P)$ of weight $0$ such that the cover $\lpt$ of $P$ on $(R^\NN, \partial_R)$ is not a Rota-Baxter operator of weight $0$.
\mlabel{pp:lr}
\end{prop}
\begin{proof}
By Corollary~\mref{co:qq0}, we only need to prove that, for each given $\omega$ as in the proposition, there is a Rota-Baxter algebra $(R,P)$ and $f,g\in R^\NN$ such that
\begin{equation}
\left(\lpt_1(\lpt(f)g)+\lpt_1(f\lpt(g))\right)-
\left(\lpt_0(f)\lpt_1(g)+\lpt_1(f)\lpt_0(g)\right)
\neq 0.
\mlabel{eq:qn01b}
\end{equation}
We will divide the proof into the eight cases of $r:=\deg\,\phi$ and $s:=\deg\,\psi$ as in Lemma~\mref{lem:cases}.

Denote $\phi(x):=\sum\limits_{i=0}^{r}a_ix^i$ and $\psi(x):=\sum\limits_{j=0}^{s}b_jx^j$. So $a_r, b_s\neq 0$.
\smallskip

\noindent
{\bf Case (i). $s> 1, r>s.$}
In Lemma~\mref{lem:cases}.(\mref{it:t1}), take $(R, P):=(\sha(\bfk)/I_1, \overline{P_{\bfk}})$, $n:=r$ and $u:=\overline{z_{0}}$. Then the lemma gives $M_r=r+(r-1)s$, $\lpt_r(f)=a_r b_s^{r-1}\overline{z_{0}}$ and $\lpt_\sigma(f)=\overline{0}$ for $\sigma< r$.  Let $g\in (\sha(\bfk)/I_{1})^\NN$ be the identity element.
Since $\lpt_\sigma(f)=\overline{0}$ for $\sigma<r$ and $r> 2$ by the assumption of Case (i), ~(\ref{eq:RBr11}) and (\ref{eq:RBl0}) give

$$\lpt_1(\lpt(f)g)=a_r\lpt_r(f)\quad \text{and}\quad \lpt_0(f)\lpt_1(g)+\lpt_1(f)\lpt_0(g)=\overline{0},$$
respectively. Also in this case, $N_f=\max\{r,\,s\}=r$ in ~(\mref{eq:Nf}). So we have $N_f<M_r$.
Then by $f_\ell=\delta_{\ell, M_r}\overline{z_{0}}$, ~(\mref{eq:RBr2}) gives $\lpt_1(f\lpt(g))=\overline{0}$.
Thus we obtain
$$\left(\lpt_1(\lpt(f)g)+\lpt_1(f\lpt(g))\right)-
\left(\lpt_0(f)\lpt_1(g)+\lpt_1(f)\lpt_0(g)\right)
=a_r\lpt_r(f)=a_r^2 b_s^{r-1}\overline{z_{0}}\neq \overline{0}.$$
This is what we need.

We use the similar argument as in Case (i) to prove other cases as follows.
\smallskip

\noindent
{\bf Case (ii). $s>1, r=s$.}
In Lemma~\mref{lem:cases}.(\mref{it:t2}), taking $(R, P):=(\sha(\bfk)/I_2, \overline{P_{\bfk}})$,
$n:=s$ and $u:=\overline{z_{1}}$ gives $M_s=s^2$, $\lpt_{s}(f)=a_rb_s^{s-1}\overline{z_{1}}+b_s^sP(\overline{z_{1}})
=a_rb_s^{s-1}\overline{z_{1}}$ and $\lpt_\sigma(f)=\overline{0}$ for $\sigma< s$. Let $g\in (\sha(\bfk)/I_{2})^\NN$ be the identity.
Then by ~(\ref{eq:RBr11}) and (\ref{eq:RBl0}), we have $$\lpt_1(\lpt(f)g)=a_r\lpt_{r}(f)+b_sP(\lpt_{s}(f))=a_r^2b_s^{s-1}\overline{z_{1}}\quad
\text{and}\quad
\lpt_0(f)\lpt_1(g)+\lpt_1(f)\lpt_0(g)=\overline{0},$$
respectively. Further by $N_f=s<M_s$ and $f_\ell=\delta_{\ell, M_s}\overline{z_{1}}$, ~(\mref{eq:RBr2}) becomes $\lpt_1(f\lpt(g))=\overline{0}$.
Thus we obtain
$$\left(\lpt_1(\lpt(f)g)+\lpt_1(f\lpt(g))\right)-
\left(\lpt_0(f)\lpt_1(g)+\lpt_1(f)\lpt_0(g)\right)
=a_r^2b_s^{s-1}\overline{z_{1}}
\neq\overline{0}.$$

\noindent
{\bf Case (iii). $s> 1, r< s.$}
In Lemma~\mref{lem:cases}.(\mref{it:t3}), take $(R, P):=(\sha(\bfk)/I_3, \overline{P_{\bfk}})$, $n:=s$ and $u:=\overline{z_{0}}$. Then $M_s=s^2$, $\lpt_{s}(f)=b_s^sP(\overline{z_{0}})
=b_s^s\overline{z_{1}}$ and $\lpt_\sigma(f)=\overline{0}$ for $\sigma< s$. Let $g\in (\sha(\bfk)/I_{3})^\NN$ be the identity. By ~(\ref{eq:RBr11}) and (\ref{eq:RBl0}), we have $$\lpt_1(\lpt(f)g)=b_sP(\lpt_{s}(f))\quad\text{and}\quad \lpt_0(f)\lpt_1(g)+\lpt_1(f)\lpt_0(g)=\overline{0},$$
respectively. Since $N_f=s<M_s$ and $f_\ell=\delta_{\ell, M_s}\overline{z_{0}}$, ~(\mref{eq:RBr2}) becomes $\lpt_1(f\lpt(g))=\overline{0}$. Thus we obtain
$$\left(\lpt_1(\lpt(f)g)+\lpt_1(f\lpt(g))\right)-
\left(\lpt_0(f)\lpt_1(g)+\lpt_1(f)\lpt_0(g)\right)
=b_sP(\lpt_s(f))=b_s^{s+1}\overline{z_{2}}\neq \overline{0}.$$

\noindent
{\bf Case (iv). $s=1, r>s.$}
We consider $(R, P):=(\sha(\bfk)/I_1, \overline{P_{\bfk}})$ and divide the proof into two subcases depending on whether or not $\sum\limits_{k=0}^{r-1} b_1^{k}$ is zero.

First assume $\sum\limits_{k=0}^{r-1} b_1^{k}\neq 0$. In Lemma~\mref{lem:cases}.(\mref{it:t4}), take $n:=r$ and $u:=\overline{z_{0}}$. Then $M_r=2r-1$, $\lpt_r(f)=\sum\limits_{k=0}^{r-1}a_r b_1^{k}\overline{z_{0}}$ and $\lpt_\sigma(f)=\overline{0}$ for $\sigma< r$. Let $g:=(g_k)\in (\sha(\bfk)/I_{1})^\NN$ be the identity. Then ~(\ref{eq:RBr11}) and (\ref{eq:RBl0}) give $$\lpt_1(\lpt(f)g)=a_r\lpt_r(f)\quad\text{and}\quad
\lpt_0(f)\lpt_1(g)+\lpt_1(f)\lpt_0(g)=\overline{0},$$
respectively. Further, by $N_f=r<M_r$ and $f_\ell=\delta_{\ell, M_r}\overline{z_{0}}$, ~(\mref{eq:RBr2}) gives $\lpt_1(f\lpt(g))=\overline{0}$. Thus we obtain
$$\left(\lpt_1(\lpt(f)g)+\lpt_1(f\lpt(g))\right)-
\left(\lpt_0(f)\lpt_1(g)+\lpt_1(f)\lpt_0(g)\right)
=a_r\lpt_r(f)=a_r^2\Big(\sum_{k=0}^{r-1} b_1^{k}\Big)\overline{z_{0}}\neq 0.$$

Next assume $\sum\limits_{k=0}^{r-1} b_1^{k}=0$. Then  $\sum\limits_{k=0}^{r-2} b_1^{k}=-b_1^{r-1}$. In Lemma~\mref{lem:cases}.(\mref{it:t4}), we can
take $n:=r-1$ and $u:=\overline{z_{0}}$. Then $M_{r-1}=2(r-1)$, $$\lpt_{r-1}(f)=\sum\limits_{k=0}^{r-2}a_r b_1^{k}\overline{z_{0}}=a_r (\sum\limits_{k=0}^{r-2}b_1^{k})\overline{z_{0}}=
-a_rb_1^{r-1}\overline{z_{0}}$$ and $\lpt_\sigma(f)=\overline{0}$ for $\sigma<r-1$. Let $g:=(g_k)\in (\sha(\bfk)/I_{1})^\NN$ with $g_k:=\delta_{k, 1}\overline{z_{0}}.$ Then by ~(\mref{eq:RBr10}), we have $$\lpt_1(\lpt(f)g)=a_r{r\choose r-1}\lpt_{r-1}(f)g_1=
-ra_r^2b_1^{r-1}\overline{z_{0}}.$$
Further, by $N_f=r\leq M_{r-1}$ and $f_\ell=\delta_{\ell, M_{r-1}}\overline{z_{0}}$, ~(\mref{eq:RBr2}) gives $$\lpt_1(f\lpt(g))=a_r f_r\lpt_0(g)=a_rf_rP(g_0)=\overline{0}.$$
By $\lpt_0(f)=\lpt_0(g)=\overline{0}$, we have $\lpt_0(f)\lpt_1(g)+\lpt_1(f)\lpt_0(g)=\overline{0}$. Thus we obtain
$$\left(\lpt_1(\lpt(f)g)+\lpt_1(f\lpt(g))\right)-
\left(\lpt_0(f)\lpt_1(g)+\lpt_1(f)\lpt_0(g)\right)=
-ra_r^2b_1^{r-1}\overline{z_{0}}\neq\overline{0}.$$

\noindent
{\bf Case (v). $s=1, r=s.$}
In Lemma~\mref{lem:cases}.(\mref{it:t5}), take $(R, P):=(\sha(\bfk)/I_1, \overline{P_{\bfk}})$, $n:=1$ and $u:=\overline{z_{0}}$. Then $M_1=1$, $\lpt_1(f)=a_1\overline{z_{0}}+b_1P(\overline{z_{0}})=a_1\overline{z_{0}}$ and $\lpt_0(f)=\overline{0}$. Let $g:=(g_k)\in (\sha(\bfk)/I_{1})^\NN$ with $g_k:=\delta_{k, 0}\overline{z_{0}}.$ So $g$ is the identity. By ~(\mref{eq:RBr11}) and $\lpt_0(f) =\overline{0}$, we have $$\lpt_1(\lpt(f)g)=a_1\lpt_{1}(f)+b_1P(\lpt_{1}(f)).$$ Since $f_\ell=\delta_{\ell, 1}\overline{z_{0}}$ and $\lpt_0(g)=P(g_0)=\overline{0}$, ~(\mref{eq:RBr2}) gives $\lpt_1(f\lpt(g))=\overline{0}$. By ~(\mref{eq:RBl0}) and $\lpt_0(f) =\overline{0}$, we have $\lpt_0(f)\lpt_1(g)+\lpt_1(f)\lpt_0(g)=\overline{0}$. Thus we obtain
$$\left(\lpt_1(\lpt(f)g)+\lpt_1(f\lpt(g))\right)-
\left(\lpt_0(f)\lpt_1(g)+\lpt_1(f)\lpt_0(g)\right)=a_1\lpt_{1}(f)+b_1P(\lpt_{1}(f))=a_1^2\overline{z_{0}}
\neq\overline{0}.$$

\noindent
{\bf Case (vi). $s=1, r<s.$}
We consider
$(R, P):=(\sha(\bfk)/I_3, \overline{P_{\bfk}})$
and divide the proof into two subcases depending on whether or not $b_1=1$.

First assume $b_1\neq 1$. In Lemma~\mref{lem:cases}.(\mref{it:t6}), take $n:=1$ and $u:=\overline{z_{0}}$. Then $M_1=1$, $\lpt_0(f)=\overline{0}$ and $\lpt_1(f)=b_1P(\overline{z_{0}})=b_1\overline{z_{1}}$. Let $g:=(g_k)\in (\sha(\bfk)/I_{3})^\NN$ be the identity, so $g_k:=\delta_{k, 0}\overline{z_{0}}.$
Then ~(\mref{eq:RBr2}) and ~(\mref{eq:RBr11}) give
$$\lpt_1(f\lpt(g))=b_1P(f_1\lpt_0(g))\quad\text{and}\quad
\lpt_1(\lpt(f)g)=b_1P(\lpt_{1}(f)),$$
respectively.
By $\lpt_0(f)=\overline{0}$,  ~(\mref{eq:RBl0}) becomes $$\lpt_0(f)\lpt_1(g)+\lpt_1(f)\lpt_0(g)=\lpt_1(f)\overline{z_{1}}.$$ Thus we obtain
$$\left(\lpt_1(\lpt(f)g)+\lpt_1(f\lpt(g))\right)-
\left(\lpt_0(f)\lpt_1(g)+\lpt_1(f)\lpt_0(g)\right)=b_1^{2}\overline{z_{2}}+b_1\overline{z_{2}}-b_1\overline{z_{1}}^2=b_1(b_1-1)\overline{z_{2}}\neq\overline{0}.$$

Next assume $b_1=1$. Let both $f:=(f_\ell)$ and $g:=(g_k)$ be the identity element of $(\sha(\bfk)/I_{3})^\NN$. Then $\lpt_0(f)=\lpt_0(g)=\overline{z_{1}}$. By ~(\mref{eq:3P1}), we have
$\lpt_{1}(f)=\lpt_{1}(g)=a_0\overline{z_0}+b_0\overline{z_{1}}.$ Then applying the commutativity of the multiplication and ~(\mref{eq:RBr11}), we have
$$\lpt_1(f\lpt(g))=\lpt_1(\lpt(f)g)=a_0\lpt_0(f)+b_0P(\lpt_0(f))+P(\lpt_{1}(f))=2a_0\overline{z_1}+2b_0\overline{z_2}.$$
Further, $\lpt_0(f)\lpt_1(g)+\lpt_1(f)\lpt_0(g)=2\overline{z_1}(a_0\overline{z_0}+b_0\overline{z_{1}})=2a_0 \overline{z_1}+4b_0\overline{z_2}$. Thus we obtain
$$\left(\lpt_1(\lpt(f)g)+\lpt_1(f\lpt(g))\right)-
\left(\lpt_0(f)\lpt_1(g)+\lpt_1(f)\lpt_0(g)\right)=2a_0\overline{z_1}\neq\overline{0}.$$

\noindent
{\bf Case (vii). $s=0, r>s.$}
In Lemma~\mref{lem:cases}.(\mref{it:t7}), take $(R, P):=(\sha(\bfk)/I_1, \overline{P_{\bfk}})$,
 $n:=r$ and $u:=\overline{z_{0}}$. Then $M_r=2r-1$, $\lpt_r(f)=a_r \overline{z_{0}}$ and $\lpt_\sigma(f)=\overline{0}$ for $\sigma<r$. Let $g:=(g_k)\in (\sha(\bfk)/I_{1})^\NN$ be the identity with $g_k:=\delta_{k, 0}\overline{z_{0}}.$ Since $\lpt_\sigma(f)=\overline{0}$ for $\sigma<r$, ~(\mref{eq:RBr11}) gives $\lpt_1(\lpt(f)g)=a_r\lpt_r(f)$.
By $N_f=r\leq M_r$ and $f_\ell=\delta_{\ell, M_r}\overline{z_{0}}$, ~(\mref{eq:RBr2}) gives $\lpt_1(f\lpt(g))=a_r f_r\lpt_0(g)=\overline{0}$. By $\lpt_0(f)=\overline{0}$,  ~(\mref{eq:RBl0}) becomes $\lpt_0(f)\lpt_1(g)+\lpt_1(f)\lpt_0(g)=\overline{0}$. Thus we obtain
\begin{equation*}
\left(\lpt_1(\lpt(f)g)+\lpt_1(f\lpt(g))\right)-
\left(\lpt_0(f)\lpt_1(g)+\lpt_1(f)\lpt_0(g)\right)
=a_r\lpt_r(f)=a_r^2\overline{z_0}\neq \overline{0}.
\end{equation*}

\noindent
{\bf Case (viii). $s=0, r=s.$}
In Lemma~\mref{lem:cases}.(\mref{it:t8}), take $(R, P):=(\sha(\bfk)/I_3, \overline{P_{\bfk}})$, $n:=1$ and $u:=\overline{z_{0}}$. Then $M_1=0$, $\lpt_1(f)=a_0\overline{z_{0}}+b_0P(\overline{z_{0}})=
a_0\overline{z_{0}}+b_0\overline{z_{1}}$ and $\lpt_0(f)=P(\overline{z_{0}})=\overline{z_{1}}$. Let $g=f$, i.e., $g:=(g_k)\in (\sha(\bfk)/I_{3})^\NN$  with $g_k:=\delta_{k, 0}\overline{z_{0}}$. Then applying the commutativity of the multiplication and ~(\mref{eq:RBr11}), we have $$\lpt_1(f\lpt(g))=\lpt_1(\lpt(f)g)=a_0\lpt_{0}(f)+b_0P(\lpt_{0}(f)).$$ Thus we obtain
\begin{eqnarray*}
&&\left(\lpt_1(\lpt(f)g)+\lpt_1(f\lpt(g))\right)-
\left(\lpt_0(f)\lpt_1(g)+\lpt_1(f)\lpt_0(g)\right)\\
&=&2(a_0\overline{z_{1}}+b_0\overline{z_{2}})
-2\overline{z_{1}}(a_0\overline{z_{0}}+b_0\overline{z_{1}}) \\
&=&-2b_0\overline{z_{2}}\neq \overline{0}.
\end{eqnarray*}

To recapitulate, applying Corollary~\mref{co:qq0}, we obtain that for each $(s, r)\in \NN\times\NN$, the given cover $\lpt$ of the chosen $P$ on $(R^\NN, \partial_R)$ is not a Rota-Baxter operator, completing the proof of Proposition~\mref{pp:lr}.
\end{proof}

\subsection{Proof of Theorem~\mref{thm:lr}.(\mref{it:weightlambda})}
\mlabel{ss:part2}
Finally we prove Theorem~\mref{thm:lr}.(\mref{it:weightlambda}).

Let $(R, P)$ be an arbitrary Rota-Baxter algebra of arbitrary weight $\lambda$. Recall from Proposition~\mref{prop:qq}.(\mref{it:pqn}) that
the cover $\lpt$ of $P$ on $(R^\NN,\partial_R)$ is again a Rota-Baxter operator of weight $\lambda$ if and only if for all $f, g\in R^\NN$, and $n\in \NN$,
\begin{equation}
\sum_{k=0}^{n}\sum_{j=0}^{n-k} {n\choose k} {n-k\choose j}
\lambda^{k}\lpt_{n-j}(f)\lpt_{k+j}(g)-\left(\lpt_n(\lpt (f)g)+\lpt_n(f\lpt (g))+\lambda \lpt_n(fg)\right)=0.
\mlabel{eq:comP0z}
\end{equation}

\noindent
(\mref{it:b1})$\Longrightarrow$(\mref{it:b2}). If Item~(\mref{it:b1}) holds, then as a special case, for every Rota-Baxter algebra $(R, P)$ of weight $0$, the cover $\lpt$ of $P$ is still a Rota-Baxter operator of weight $0$. So by Theorem~\mref{thm:lr}.(\mref{it:weight0}), $\omega$ is in $\calt_0$, that is, $\omega=xy-a_0$ or $\omega=xy-(b_0y+yx)$.

First consider $\omega=xy-a_0$. Then ~(\mref{eq:lptn1}) gives $\lpt_{n}(f)=a_0f_{n-1}$ for all $f\in R^\NN, n\in\NN_+.$ Together with $\lpt_0(f)=P(f_0)$, we obtain
\begin{equation}
\lpt(f)=(P(f_0), a_0f_0, a_0f_1, \cdots).
\mlabel{eq:4PP}
\end{equation}
We take $(R, P):=(\sha(\bfk)/I_{1}, \overline{P_{\bfk}})$ of weight $\lambda$, and $f=g\in(\sha(\bfk)/I_{1})^\NN$ with
$f_{\ell}:=\delta_{\ell, 0}\overline{z_{0}}$.
Applying ~(\mref{eq:4PP}), we have
\begin{equation}
\lpt_{1}(f)=\lpt_{1}(g)=a_0\overline{z_{0}},\quad \lpt_{0}(f)=\lpt_{0}(g)=P(\overline{z_{0}})=\overline{0}.
\label{eq:cc0}
\end{equation}
Then we obtain
\begin{eqnarray*}
\sum_{k=0}^{1}\sum_{j=0}^{1-k} {1\choose k} {1-k\choose j}
\lambda^{k}\lpt_{1-j}(f)\lpt_{k+j}(g)&=&
\lpt_0(f)\lpt_1(g)+\lpt_1(f)\lpt_0(g)+\lambda \lpt_1(f)\lpt_1(g)\notag\\
&=&\lambda a_0^2\overline{z_{0}}\quad \text{(by ~(\mref{eq:cc0}))}
\end{eqnarray*}
and
\begin{eqnarray*}
\lpt_1(\lpt(f)g)+\lpt_1(f\lp(g))+\lambda \lpt_1(fg)
&=&a_0 \lpt_{0}(f)g_0 +a_0f_0\lpt_{0}(g)+\lambda a_0 f_0g_0\quad \text{(by ~(\mref{eq:4PP}))}\notag\\
&=&\lambda a_0\overline{z_{0}}.\quad \text{(by ~(\mref{eq:cc0}))}\notag
\end{eqnarray*}
Then
$$\sum_{k=0}^{1}\sum_{j=0}^{1-k} {1\choose k} {1-k\choose j}
\lambda^{k}\lpt_{1-j}(f)\lpt_{k+j}(g)-\left(\lpt_1(\lpt(f)g)+
\lpt_1(f\lpt(g))+\lambda \lpt_1(fg)\right)=\lambda a_0(a_0-1) \overline{z_{0}}.$$
Thus for a given nonzero $\lambda$, ~(\mref{eq:comP0z}) holds in the case of $n=1$ for the above chosen Rota-Baxter algebra $(\sha(\bfk)/I_1, \overline{P_\bfk})$ and $f, g\in \sha(\bfk)/I_1$
if and only if $a_0=0$ or $a_0=1$, i.e., $\omega=xy$ or $\omega=xy-1$.
Next consider $\omega=xy-(b_0y+yx)$. Then applying ~(\mref{eq:lptn1}) gives
\begin{equation}
\lpt_1(f)=b_0P(f_0)+P(f_1)\quad\text{for all} \  f\in R^\NN.
\mlabel{eq:bP1}
\end{equation}
We take $(R, P):=(\sha(\bfk)/I_{2}, \overline{P_{\bfk}})$, and $f,  g \in(\sha(\bfk)/I_{2})^\NN$ with $f_{\ell}:=\delta_{\ell, 1}\overline{z_{0}}, g_k:=\delta_{k, 0}\overline{z_{0}}.$
Then we obtain
\begin{equation}
\lpt_{0}(f)=\overline{0},\quad\lpt_{0}(g)=\overline{z_1},\quad
\lpt_{1}(f)=P(f_1)=\overline{z_{1}}, \quad\lpt_{1}(g)=b_0P(g_0)=b_0\overline{z_{1}}.
\mlabel{eq:cc1}
\end{equation}
Thus
\begin{eqnarray*}
\sum_{k=0}^{1}\sum_{j=0}^{1-k} {1\choose k} {1-k\choose j}
\lambda^{k}\lpt_{1-j}(f)\lpt_{k+j}(g)&=&\lpt_0(f)\lpt_1(g)+\lpt_1(f)\lpt_0(g)+
\lambda \lpt_1(f)\lpt_1(g)\\
&=&\overline{z_{1}}^2+\lambda b_0\overline{z_{1}}^2=\lambda\overline{z_{1}}+\lambda^2 b_0\overline{z_{1}}
\end{eqnarray*}
and
\begin{eqnarray*}
&&\lpt_1(\lpt (f)g)+\lpt_1(f\lpt (g))+\lambda \lpt_1(fg)\\
&=&\lpt_1(\lpt (f))+\lpt_1(f\lpt (g))+\lambda \lpt_1(f)\quad \text{(since $g$ is the identity element)}\\
&=&b_0P(\lpt_0(f))+P(\lpt_1(f))
+b_0P((f\lpt (g))_0)+P((f\lpt (g))_1)\\
&&+\lambda \left(b_0P(f_0)+P(f_1)\right)\quad \text{(by ~(\mref{eq:bP1}))}\\
&=&b_0P(\lpt_0(f))+P(\lpt_1(f))
+b_0P(f_0 \lpt_0(g))+P\left(f_1 \lpt_0(g)+f_0 \lpt_1(g)+\lambda f_1 \lpt_1(g)\right)\\
&&+\lambda \left(b_0P(f_0)+P(f_1)\right)\quad \text{(by ~(\mref{eq:hurprod}))}\\
&=&\lambda \overline{z_{1}}\quad \text{(by ~(\mref{eq:cc1}))}.
\end{eqnarray*}
Thus for a nonzero $\lambda\in\bfk$, ~(\mref{eq:comP0z}) holds in the case of $n=1$ for the above chosen Rota-Baxter algebra and $f, g$ if and only if
$\lambda \overline{z_1}+\lambda^2 b_0 \overline{z_1}-\lambda \overline{z_1}=\lambda^2 b_0\overline{z_{1}}= \overline{0}$,
which holds if and only if $b_0=0$. Thus $\omega=xy-yx$.

Therefore, $\omega$ must be in $\calt_{\bfk}=\{xy, xy-1, xy-yx\}.$

\noindent
(\mref{it:b2})$\Longrightarrow$(\mref{it:b1}). For $\omega=xy-1$, the cover $\lpt$ of $P$ on $(R^\NN,\partial_R)$ is again a Rota-Baxter operator of weight $\lambda$ by~\cite[Proposition 3.8]{ZGK}.
When $\omega=xy$ or $\omega=xy-yx$, we need to show that the cover $\lpt$ of $P$ on $(R^\NN,\partial_R)$ satisfies ~(\mref{eq:comP0z}).

For $\omega=xy$,  applying ~(\mref{eq:lptn1}), we have
\begin{equation}
\lpt_n(f)=0 \quad \text{for all}\  n\in \NN_+, f\in R^\NN.
\mlabel{eq:comP0}
\end{equation}
By Proposition~\mref{prop:qq}.(\mref{it:pq0}), ~(\mref{eq:comP0z}) holds for $n=0$.
If $n\in\NN_+$, then the maximum of the subscripts $m$ of the expressions $\lpt_m$ appearing in each term of the left side of ~(\mref{eq:comP0z}) is strictly larger than $0$ and then by ~(\mref{eq:comP0}), each term in the left side of ~(\mref{eq:comP0z}) is $0$. Thus ~(\mref{eq:comP0z}) holds for all $n\in\NN_+$.

For $\omega=xy-yx$, by ~(\mref{eq:lptn1}), the cover $\lpt$ of $P$ is given by
\begin{equation}
\lpt_n(f)=\lpt_{n-1}(\partial_R f)=\lpt_0(\partial_R^n f)=P(f_n) \quad \text{for all}\  n\in \NN, f\in R^\NN.
\mlabel{eq:comP}
\end{equation}
Furthermore, for all $f, g\in R^\NN$, and $n\in \NN$,
\begin{eqnarray*}
&&\sum_{k=0}^{n}\sum_{j=0}^{n-k} {n\choose k} {n-k\choose j}
\lambda^{k}\lpt_{n-j}(f)\lpt_{k+j}(g)\\
&=&\sum_{k=0}^{n}\sum_{j=0}^{n-k} {n\choose k} {n-k\choose j}
\lambda^{k}P(f_{n-j})P(g_{k+j})\quad \text{(by ~(\mref{eq:comP}))}\\
&=&\sum_{k=0}^{n}\sum_{j=0}^{n-k} {n\choose k} {n-k\choose j}
\lambda^{k}\left(P(P(f_{n-j})g_{k+j})+P(f_{n-j}P(g_{k+j}))+
\lambda P(f_{n-j} g_{k+j})\right)\quad \text{(by ~(\mref{eq:bax1}))}\\
&=&\sum_{k=0}^{n}\sum_{j=0}^{n-k} {n\choose k} {n-k\choose j}
\lambda^{k}\left(P(\lpt_{n-j}(f)g_{k+j})+P(f_{n-j}\lpt_{k+j}(g))+
\lambda P(f_{n-j} g_{k+j})\right)\quad \text{(by ~(\mref{eq:comP}))}\\
&=&P((\lpt (f)g)_n)+P((f\lpt (g))_n)+\lambda P((fg)_n)\quad \text{(by ~(\mref{eq:hurprod}))}\\
&=&\lpt_n(\lpt (f)g)+\lpt_n(f\lpt (g))+\lambda \lpt_n(fg).\quad \text{(by ~(\mref{eq:comP}))}
\end{eqnarray*}
Then ~(\mref{eq:comP0z}) holds.

Now we have completed the proof of Theorem~\mref{thm:lr}.

\smallskip

\noindent
{\bf Acknowledgements}:
This work is supported by the National Natural Science Foundation of China (Grant No. 11771190) and the China Scholarship Council (Grant No. 201606180084). Shilong Zhang thanks Rutgers University-Newark for its hospitality during his visit from August 2016 to August 2017.

\end{document}